\documentclass[preprint,12pt,sort&compress]{elsarticle}

\usepackage{graphicx}
\usepackage{subcaption}
\usepackage{placeins}
\usepackage{physics, amsmath, amssymb}

\usepackage{amsthm}
\newtheorem{proposition}{Proposition}
\newtheorem{definition}{Definition}

\usepackage{pgfplots, pgfplotstable}
\pgfplotsset{
    compat = 1.9, 
    xlabel style = {font=\footnotesize, yshift = 1ex},
    ylabel style = {font=\footnotesize, yshift = -1ex},
    xticklabel style = {font=\scriptsize},
    yticklabel style = {font=\scriptsize},
    legend style = {font=\scriptsize},
    title style = {font=\normalsize, yshift = -1ex},
    minor tick num = 1,
    grid = both,
    major grid style = {lightgray},
    minor grid style = {lightgray!40},
    width=0.45\textwidth,
}
\usepackage{tikz}
\usepackage{xifthen} 
\newcommand{\myGlobalTransformation}[2]
{
    \pgftransformcm{1}{0}{0.8}{0.4}{\pgfpoint{#1cm}{#2cm}}
}

\usetikzlibrary{shapes,backgrounds,calc}
\makeatletter
\tikzset{circle split part fill/.style  args={#1,#2}{%
    alias=tmp@name,%
    postaction={%
     insert path={\pgfextra{%
      \pgfpointdiff{\pgfpointanchor{\pgf@node@name}{center}}%
      {\pgfpointanchor{\pgf@node@name}{east}}%
      \pgfmathsetmacro\insiderad{\pgf@x}%
      \fill[#1] (\pgf@node@name.base) ([xshift=-\pgflinewidth]\pgf@node@name.east) arc (0:180:\insiderad-\pgflinewidth)--cycle;
      \fill[#2] (\pgf@node@name.base) ([xshift=\pgflinewidth]\pgf@node@name.west) arc (180:360:\insiderad-\pgflinewidth)--cycle;
}}}}}
\makeatother

\usepackage{hyperref}
\hypersetup{
    colorlinks=true,
    linkcolor=blue,
    filecolor=magenta,      
    urlcolor=cyan,
}

\usepackage{comment} 

\usepackage{etoolbox}
\patchcmd{\emailauthor}{(#2)}{}{}{}
\patchcmd{\urlauthor}{(#2)}{}{}{}

\date{April 11, 2022}

\begin{document}

\begin{frontmatter}

\title{Estimating Discretization Error with Preset Orders of Accuracy and Fractional Refinement Ratios \tnoteref{funding}}
\author[label1]{Sharp Chim Yui Lo\fnref{fn1}}
\fntext[fn1]{Email Address: chim.yui.lo@alumni.ethz.ch}
\affiliation[label1]{
    organization={Institute for Atmospheric and Climate Science},
    addressline={ETH Zurich},
    postcode={8092 Zurich},
    country={Switzerland}
}
\tnotetext[funding]{This research did not receive any specific grant from funding agencies in the public, commercial, or not-for-profit sectors.}
   


\begin{abstract}
Verification of solutions is crucial for establishing the reliability of simulations. A central challenge is to find an accurate and reliable estimate of the discretization error. Current approaches to this estimation rely on the observed order of accuracy; however, studies have shown that it may alter irregularly or become undefined. Therefore, we propose a grid refinement method which adopts constant orders given by the user, called the Preset Orders Expansion Method (POEM). The user is guaranteed to obtain the optimal set of orders through iterations and hence an accurate estimate of the discretization error. This method evaluates the reliability of the estimation by assessing the convergence of the expansion terms, which is fundamental for all grid refinement methods. We demonstrate these capabilities using advection and diffusion problems along different refinement paths. POEM requires a lower computational cost when the refinement ratio is higher. However, the estimated error suffers from higher uncertainty due to the reduced number of shared grid points. We circumvent this by using fractional refinement ratios and the Method of Interpolating Differences between Approximate Solutions (MIDAS). As a result, we can obtain a global estimate of the discretization error of lower uncertainty at a reduced computational cost.
\end{abstract}


\end{frontmatter}



\section{Introduction} \label{sec:intro}
As numerical simulation becomes an essential part of decision-making in science and engineering, verification and validation \citep{Roache1998VerificationEngineering,Stern2001ComprehensiveProcedures,OberkampfRoy2010} are widely used to establish the reliability of the approximate solutions given by simulations. Verification intends to show that one is ``solving the equations right'', whereas validation intends to show that one is ``solving the right equations'' \citep{Roache1998article}. In particular, solution verification is the process of estimating the numerical errors in an approximate solution for which the exact solution is unknown. Among all sources of numerical errors, the discretization error (DE) is usually the largest and most difficult to estimate \citep{Roy2005,Eca2014AStudies}. The estimation of DE in solution verification is the focus of this paper.

The estimation of DE by grid refinement methods \citep{Roy2010} has received substantial interest in computational physics \citep{Roache1994Perspective:Studies,Eca2014AStudies,Phillips2016ADynamics,Rider2016RobustAnalysis}. The main advantage of these methods is that they can be applied to any system response quantity, on any discretization method, and in a post-processing manner \citep{Stern2001ComprehensiveProcedures,Roy2010}. In principle, these methods use approximate solutions on a sequence of grids to obtain an estimate of the exact solution, followed by comparing it with the approximate solutions to obtain an estimate of the DE. This estimation is reliable if the approximate solutions lie in the so-called \textit{asymptotic range} \citep{Salas2006,Xing2010,Roy2010,Phillips2016ADynamics,Eca2018}, which is defined as the sequence of grids over which the DE reduces in the formal (theoretical) order of accuracy \citep{Roy2010}.

Most of the current grid refinement methods assess the reliability of the estimated DE by comparing the formal order of accuracy with the observed order of accuracy \citep{Roy2010,Orozco2010VerificationSolutions}; however, the use of the observed order has two problems. First, a small variation of the observed order may cause both the magnitude and the uncertainty of the estimated DE to change irregularly \citep{Roy2003,Eca2009,Xing2010,Hodis2012GridAneurysms}. Second, in certain cases the observed order is not well defined \citep{Roy2003,Phillips2014RichardsonDynamics,Phillips2016ADynamics}. These problems are usually overcome by placing a limit on the observed order, but then the estimated DE may contain unpredictable errors \citep{Phillips2014RichardsonDynamics,Phillips2016ADynamics}.

On the other hand, grid refinement methods generally suffer from a high computational cost due to the large number of grid points required to achieve the asymptotic range \citep{Eca2018}. One possible way to lower the cost is using refinement ratios greater than 0.5 \citep{Roy2010}. This has been applied to estimate local errors in critical locations of the domain \citep{Roache1994Perspective:Studies,Trivedi2019}, but there has been less evidence for such applications throughout the domain to estimate global errors. A central challenge is that the number of shared grid points across refinement levels decreases as the refinement ratio increases, which in turn increases the statistical uncertainty of the global estimates.

In this paper, we propose an alternative grid refinement method, called the Preset Orders Expansion Method (POEM). This method avoids the above problems with the observed order by using constant orders given by the user. In addition, it can be used to assess the asymptotic convergence of approximate solutions. While the idea of using constant orders is not new \citep{Eca2002}, no study to date has applied it to estimate the magnitude of DE. We also study the use of fractional refinement ratios greater than 0.5 to reduce the computational demand of POEM. The reduction in the number of shared grid points due to fractional refinement is overcome by a technique called Method of Interpolating Differences between Approximate Solutions (MIDAS).

\section{Discretization Error} \label{sec:DiscErr}
Physical and engineering systems are often described by a mathematical model, for example, a set of partial differential equations with initial and boundary conditions. Such a model is discretized when being solved on digital computers. The numerical error associated with the discretization process is known as \textit{discretization error} (DE).

DE, denoted by $\varepsilon$, is defined as the difference between the exact solution to the discretized model $\phi$ and the exact solution to the mathematical model $\phi_{e}$ \citep{Roy2010}. This can be formulated as follows.
\begin{equation} \label{eq:DE}
    \varepsilon \equiv \phi - \phi_{e}.
\end{equation}
In solution verification, $\phi_{e}$ is not available, so the DE can only be estimated. An estimate of DE, denoted by $\Tilde{\varepsilon}$, is defined as the difference between $\phi$ and an estimate of $\phi_{e}$, denoted by $\Tilde{\phi}_{e}$. This can be formulated as follows.
\begin{equation} \label{eq:est_DE}
    \Tilde{\varepsilon} \equiv \phi - \Tilde{\phi}_{e}.
\end{equation}
The solution $\phi$ is termed \textit{approximate solution}: the exact solution to the discretized model can be understood as an approximation to the mathematical model.

One of the assumptions of grid refinement methods is that DE constitutes the majority of numerical errors in an approximate solution \citep{Roy2010}. This is valid for most practical applications, since other sources of numerical errors can be easily eliminated or quantified \citep{Roy2005,Eca2014AStudies}. The calculations in this study are based on this assumption.

\section{Preset Orders Expansion Method (POEM)} \label{sec:POEM}
In this section, we explain the principles of POEM and demonstrate its applications on a few typical problems. We first present the principal equation and the procedures for estimating DE in this method. Then, we address how to choose the preset orders and how to adapt the principal equation to different refinement paths. Lastly, we discuss the computational demand of this method.


\subsection{The Principal Equation}
When $\phi_e$ is smooth over the domain of the mathematical model, it can be related to $\phi$ by a series expansion in orders of a grid spacing parameter $h$ \cite{Roy2010}:
\begin{equation} \label{eq:model_infty}
    \phi = \phi_{e} + \sum_{m=1}^{\infty} C_{q_m} h^{q_m}.
\end{equation}
Here, the coefficients $\{C_{q}\}$ are functions of space and time but independent of $h$. The orders of the coefficients $\{q_m: q_1 < q_2 < \dots \}$ are integers determined by the numerical scheme being used. In particular, $q_1$ is known as \textit{formal order of accuracy} \citep{Roy2010}.

In POEM, this infinite series is modeled by the following principal equation.
\begin{equation} \label{eq:model_orig}
    \phi = \Tilde{\phi}_{e} + \sum_{m=1}^{k} C_{p_m} h^{p_m},
\end{equation}
where $k \ge 2$ and $\{p_m: p_1 < p_2 < \dots \}$ are constant values given by the user. Preset orders $\{p_m\}$ do not necessarily match the correct ones in $\{q_m\}$. In this equation the unknowns are $\Tilde{\phi}_e$ and $\{C_{p}\}$, whereas in many other grid refinement methods the unknowns include an exponent of $h$, called the \textit{observed order of accuracy} \cite{Roy2010,Orozco2010VerificationSolutions}. Because of this, the problems with the observed orders mentioned in Section \ref{sec:intro} are avoided.


\subsection{Estimation of Discretization Error} \label{subsec:est_DE}
Solving for $\Tilde{\phi}_e$ and $\{C_{p}\}$ uniquely requires $k + 1$ approximate solutions orthogonal to each other. Such approximate solutions can be obtained by performing the same simulation on a set of systematically refined grids \citep{Roy2010}. This set of approximate solutions, denoted by $\{\phi_l\}$, forms a system of equations.
\begin{equation} \label{eq:model_system}
    \phi_l = \Tilde{\phi}_{e} + \sum_{m=1}^{k} C_{p_m} (r^{l-1} h)^{p_m},
\end{equation}
where $l \in [1,k+1]$ denotes the $l$th refinement level, $r \in [0.5,1)$ is the refinement ratio between two adjacent grid levels. The solutions to this system are $\Tilde{\phi}_e$ and $\{C_{p}h^p\}$.

For example, when $k = 2$, the principal equation (Equation \ref{eq:model_orig}) reads
\begin{equation} \label{eq:model_k2}
    \phi = \Tilde{\phi}_{e} + C_{p_1} h^{p_1} + C_{p_2} h^{p_2}.
\end{equation}
The system of equations (Equation \ref{eq:model_system}) can be written as
\begin{align} \label{eq:sys3eq_t}
    \begin{bmatrix}
        1 & 1 & 1 \\
        1 & r^{p_1} & r^{p_2} \\
        1 & r^{2p_1} & r^{2p_2}
    \end{bmatrix}
    \begin{bmatrix}
        \Tilde{\phi}_{e} \\
        C_{p_1}h^{p_1} \\
        C_{p_2}h^{p_2}
    \end{bmatrix}
    &=
    \begin{bmatrix}
        \phi_1 \\
        \phi_2 \\
        \phi_3
    \end{bmatrix}.
\end{align}
The analytical solutions of the unknowns are as follows.
\begin{align}
    \Tilde{\phi}_{e} &= \phi_1 - C_{p_1} h^{p_1} - C_{p_2} h^{p_2}, \\
    C_{p_1} h^{p_1} &= \frac{r^{p_2} e_{21} - e_{32}}{r^{p_1} (1 - r^{p_1}) (1 - r)}, \label{eq:C_p1} \\
    C_{p_2} h^{p_2} &= \frac{e_{32} - r^{p_1} e_{21}}{r^{p_1} (1 - r^{p_2}) (1 - r)} \label{eq:C_p2},
\end{align}
where $e_{ij} \equiv \phi_i - \phi_j$ are the \textit{differences between approximate solutions}.

By comparing $\phi_l$ with $\Tilde{\phi}_e$ at the same location, we can obtain a local estimate of the DE, $\Tilde{\varepsilon}$ (see Equation \ref{eq:est_DE}). To further quantify $\Tilde{\varepsilon}$ over the entire simulation domain, we evaluate its $L_1$-, $L_2$-, and $L_{\infty}$-norms. These norms are defined respectively as
\begin{align}
    ||f||_1 &\equiv \frac{1}{N}\sum_{i=1}^{N}|f(x_i)|, \label{eq:L1_norm} \\
    ||f||_2 &\equiv \sqrt{\frac{1}{N}\sum_{i=1}^{N}|f(x_i)|^2}, \label{eq:L2_norm} \\
    ||f||_{\infty} &\equiv \max_{1 \leq i \leq N} |f(x_i)|, \label{eq:L8_norm}
\end{align}
where $f:\{x_1,\ \dots \ ,x_N\} \rightarrow \mathbb{R}$ is a discrete function on a domain of $N$ grid points.


\subsection{Assessment of Asymptotic Convergence of Numerical Solutions}
An advantage of POEM is that it can be used to examine the asymptotic convergence of approximate solutions. From the definition of an asymptotic range \cite{Roy2010}, one can deduce that the approximate solutions lie in the asymptotic range if the higher-order coefficient terms dominate the lower-order coefficient terms. We quantify this by measuring the $L_2$-norms (see Equation \ref{eq:L2_norm}) of the coefficient terms.

In the vicinity of the asymptotic range, it suffices to compare only the first two terms ($k = 1,2$), where their ratio is given by
\begin{equation} \label{eq:beta_tilde}
    \Tilde{\beta} \equiv \frac{||C_{p_2} h^{p_2}||_2}{||C_{p_1} h^{p_1}||_2}.
\end{equation}
If $\Tilde{\beta}$ is smaller than a certain threshold $\beta \in [0,1]$, approximate solutions are considered to be in the asymptotic range and vice versa. To regularize such a criterion in simulation problems, we use a fixed value of $\beta = 0.01$ throughout this study.


\subsection{Choice of Preset Orders} \label{subsec:choice_orders}
The estimated DE is accurate if the estimated exact solution $\Tilde{\phi}_e$ converges to the exact solution $\phi_e$ faster than the approximate solution $\phi$ as the grid is refined. This amounts to eliminating the first coefficient terms of $\phi$ in Equation \ref{eq:model_infty}. In POEM, this is achieved by prescribing their actual orders $q_1,\ \dots \ , q_k$ in the preset orders. In principle, actual orders are determined solely by the numerical scheme; therefore, they can be identified during code verification \citep{Knupp2002,Salari2000}, a process recommended to take place before solution verification \citep{Roy2005}. However, it is sometimes tricky to interpret the orders observed in this process due to their high sensitivity to errors \citep{Love2013}. In the following, we present a straightforward way to obtain actual orders using POEM.


\subsubsection{Guaranteed for the Actual Orders of Coefficient Terms} \label{subsubsec:guarantee}
By substituting Equation \ref{eq:model_infty} into the analytical solution of Equation \ref{eq:model_system}, such as Equation \ref{eq:C_p1} and \ref{eq:C_p2}, the following results can be obtained. If all actual orders are prescribed, that is, $p_m \equiv q_m \ \forall \ m \in [1,k]$, then
\begin{align} \label{eq:correctPreset}
    C_{p_m} h^{p_m} = C_{q_m} h^{q_m} + \mathcal{O}(h^{q_{k+1}}) \hspace{2em} \forall \ m \in [1,k].
\end{align}
In other words, every extracted coefficient term takes the desired coefficient term. If some of the actual orders are not prescribed, then
\begin{alignat}{2}
    C_{p_m} h^{p_m} &= \mathcal{O}(h^{q_m}) &&\hspace{2em} \forall \ p_m < \mu, \label{eq:wrongPreset_small} \\ 
    C_{p_m} h^{p_m} &= \mathcal{O}(h^{\mu}) &&\hspace{2em} \forall \ p_m > \mu, \label{eq:wrongPreset_large}
\end{alignat}
where $\mu \in \{q_1,\ \dots \ , q_k\}$ is the smallest missed order. We note that $q_m$ does not necessarily equal to $p_m$ since we have defined $\{p_m\}$ and $\{q_m\}$ to be ordered sets.

These results suggest that if the entire set of preset orders is correct, then every extracted coefficient term will converge at the rate of its preset order; on the contrary, if there is a missed actual order $\mu$, then some of the extracted terms will converge at a different rate than its preset order. Therefore, by replacing the wrong preset orders with $\mu$ in repeated applications of POEM, the user is guaranteed to find all the actual orders $q_1,\ \dots \ , q_k$. With these orders substituted into the preset orders, the estimated exact solution $\Tilde{\phi}_e$ will possess the highest possible order of accuracy.


\subsubsection{Realizing Correct Orders from Wrong Orders} \label{subsubsec:realize_correct}
Here we show the differences in the extracted coefficient terms when the correct or a wrong set of preset orders is used. We consider the following initial-boundary value problem: solving the linear advection equation of $\phi(x,t)$,
\begin{equation}
    \frac{\partial \phi}{\partial t} + a \frac{\partial \phi}{\partial x} = 0,
\end{equation}
in the periodic domain $x \in [0,1]$ at time $t=2$ with the initial condition $\phi(x,0) = 2 + \cos(2 \pi x)$ and the speed of advection $a = 0.5$. In this demonstration, two discretization schemes are used. The first is the Beam-Warming (BW) scheme, which is first-order-accurate in time ($q_1 = 1$) when only the time dimension is refined \citep{Love2013}. Another is the second-order Runge-Kutta, second-order upwind (RK2U2) scheme, which is second-order-accurate in time ($q_1 = 2$).

We adopt the principal equation in Equation \ref{eq:model_k2} with $h \equiv \Delta t$. To obtain a set of systematically-refined grids, we refine the coarsest grid of size $\Delta x = \Delta t = 0.01$ in the time dimension with $r = 0.5$. We solve the system of equations in Equation \ref{eq:sys3eq_t} numerically.

Let us first consider the case where the preset orders are correct, that is, $p_1 = q_1$ and $p_2 = q_2$. For the case of the BW scheme, we set $p_1 = 1$ and $p_2 = 2$. The $L_2$-norm of the extracted coefficient terms and their convergence rates are plotted in Figure \ref{fig:cNorm-1D-BW-t-correct}, where $\Delta t$ corresponds to the coarsest level in each set of three refinement levels (see Equation \ref{eq:model_system}). In the right figure, we observe that the coefficient terms converge at the rate of their respective preset order. We find similar results (not shown) in the case of the RK2U2 scheme, where $p_1 = 2$ and $p_2 = 3$ are set. These results are consistent with Equation \ref{eq:correctPreset}.

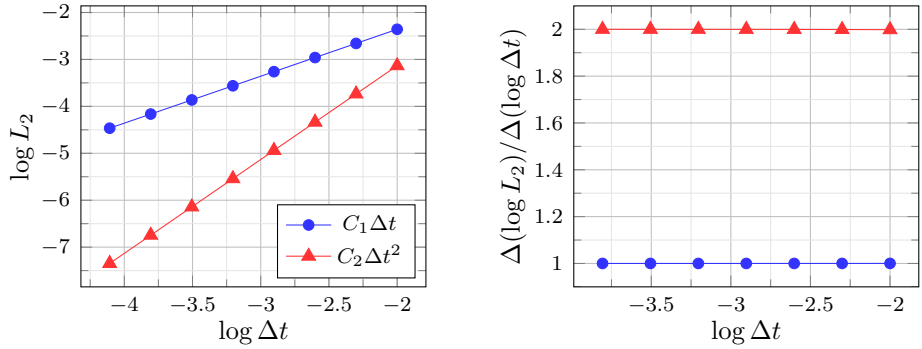
\begin{figure}[!htb]
\centering
\begin{tikzpicture}
\pgfplotstableread[skip first n=2]{figures/correct_orders/cNorm-1D-BW-t-correct.dat}{\table}
\begin{axis}[
    xtick distance = 0.5,
    ytick distance = 1,
    xlabel = $\log \Delta t$,
    ylabel = $\log L_2$,
    legend pos = south east
]
\addplot[color=blue!80, mark=*, mark size=2] table [x index=0, y index=1] {\table};
\addplot[color=red!80, mark=triangle*, mark size=3] table [x index=0, y index=2] {\table};
\legend{$C_{1} \Delta t$, $C_{2} \Delta t^{2}$}
\end{axis}
\end{tikzpicture}
\hskip 20pt
\begin{tikzpicture}
\pgfplotstableread{figures/correct_orders/cNorm_slope-1D-BW-t-correct.dat}{\table}
\begin{axis}[
    xtick distance = 0.5,
    ytick distance = 0.2,
    xlabel = $\log \Delta t$,
    ylabel = $\Delta (\log L_2) / \Delta (\log \Delta t)$
]
\addplot[color=blue!80, mark=*, mark size=2] table [x index=0, y index=1] {\table};
\addplot[color=red!80, mark=triangle*, mark size=3] table [x index=0, y index=2] {\table};
\end{axis}
\end{tikzpicture}
\caption{The $L_2$-norm of coefficient terms and their convergence rates obtained by using the correct orders, 1 and 2, in the principal equation when the $t$ dimension is refined. Approximate solutions are obtained by solving the one-dimensional advection equation using the BW scheme.}
\label{fig:cNorm-1D-BW-t-correct}
\end{figure}

Next, we consider the case in which a wrong set of preset orders is chosen. For the case of the BW scheme, we set $p_1 = 2$ and $p_2 = 3$. The results are shown in Figure \ref{fig:cNorm-1D-BW-t-wrong}. In this case, the coefficient terms do not converge at the rate of their respective preset order but a rate of 1. For the case of the RK2U2 scheme, we set $p_1 = 1$ and $p_2 = 2$. The convergence rates of $C_{1} \Delta t$ and $C_{2} \Delta t^2$ are found to be 3 and 2 respectively (not shown). These results agree with Equation \ref{eq:wrongPreset_small} and \ref{eq:wrongPreset_large}.

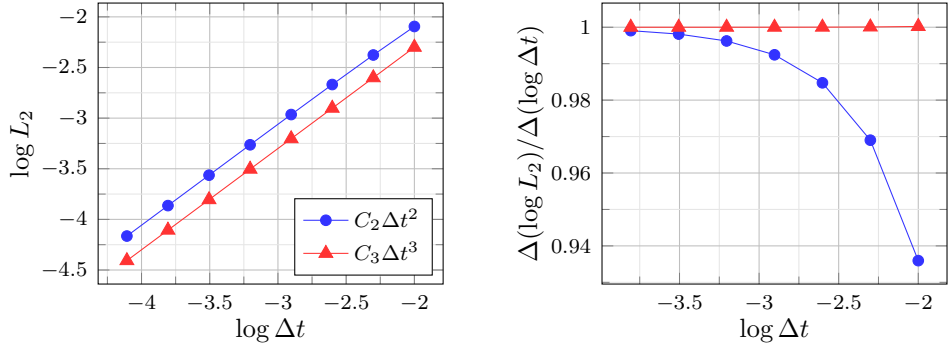
\begin{figure}[!htb]
\centering
\begin{tikzpicture}
\pgfplotstableread[skip first n=2]{figures/wrong_orders/cNorm-1D-BW-t-wrong.dat}{\table}
\begin{axis}[
    xtick distance = 0.5,
    ytick distance = 0.5,
    xlabel = $\log \Delta t$,
    ylabel = $\log L_2$,
    legend pos = south east
]
\addplot[color=blue!80, mark=*, mark size=2] table [x index=0, y index=1] {\table};
\addplot[color=red!80, mark=triangle*, mark size=3] table [x index=0, y index=2] {\table};
\legend{$C_{2} \Delta t^{2}$, $C_{3} \Delta t^{3}$}
\end{axis}
\end{tikzpicture}
\hskip 20pt
\begin{tikzpicture}
\pgfplotstableread{figures/wrong_orders/cNorm_slope-1D-BW-t-wrong.dat}{\table}
\begin{axis}[
    xtick distance = 0.5,
    ytick distance = 0.02,
    xlabel = $\log \Delta t$,
    ylabel = $\Delta (\log L_2) / \Delta (\log \Delta t)$
]
\addplot[color=blue!80, mark=*, mark size=2] table [x index=0, y index=1] {\table};
\addplot[color=red!80, mark=triangle*, mark size=3] table [x index=0, y index=2] {\table};
\end{axis}
\end{tikzpicture}
\caption{The $L_2$-norm of coefficient terms and their convergence rates obtained by using the wrong orders 2 and 3 in the principal equation when the $t$ dimension is refined. The approximate solutions are the same as those used in Figure \ref{fig:cNorm-1D-BW-t-correct}.}
\label{fig:cNorm-1D-BW-t-wrong}
\end{figure}


\FloatBarrier
\subsection{Adaptations to Different Refinement Paths} \label{subsec:adapt_paths}
When multiple dimensions are refined, the coefficient terms $\{C_p h^p\}$ in Equation \ref{eq:model_infty} can be expressed as a sum of multiple terms using the Taylor expansion. It is not obvious whether the form of Equation \ref{eq:model_orig} can be recovered and applied. This motivates us to do the following derivation.

To begin with, we write the explicit form of the coefficient terms in Equation \ref{eq:model_infty}. For the moment, we focus on refinement in the time dimension and a space dimension. Suppose the formal order of accuracy of the numerical scheme is $q_x$ when the $x$ dimension is refined and $q_t$ when the $t$ dimension is refined. Then, the equation can be written as follows.
\begin{equation} \label{eq:model_xt}
    \phi = \phi_{e} + \sum_{m=q_x}^{\infty} C_{m,0}\Delta x^{m} + \sum_{m=q_t}^{\infty} C_{0,m}\Delta t^{m} + \sum_{m=q_x+q_t}^{\infty} \sum_{n=q_t}^{m-q_x} C_{m-n,n}\Delta x^{m-n}\Delta t^{n}.
\end{equation}
As an example, the equation for the RK2U2 scheme ($q_x=q_t=2$) is
\begin{align} \label{eq:model_RK2U2_infty}
\begin{split}
    \phi &= \phi_{e} \\
    &+ C_{2,0}\Delta x^{2} + C_{3,0}\Delta x^{3} + \dots \\
    &+ C_{0,2}\Delta t^{2} + C_{0,3}\Delta t^{3} + \dots \\
    &+ C_{2,2}\Delta t^{2}\Delta x^{2} + C_{2,3}\Delta x^{2}\Delta t^{3} + \dots.
\end{split}
\end{align}

If we now truncate the series and try to form a system of equations like Equation \ref{eq:model_system}, the resulting system will be singular. The reason is that the refinement ratios for all dimensions are interdependent in a systematic grid refinement \citep{Roy2010}, which in turn lowers the degrees of freedom of that system. Such a dependency corresponds to the so-called \textit{refinement path}, which is of the form $r_t = (r_x)^s$ with $s \in \mathbb{R}$ in this case.

To obtain a non-singular system from Equation \ref{eq:model_xt}, we need to group the coefficient terms that depend on each other. After that, we obtain
\begin{align} \label{eq:model_path_infty}
    \phi = \phi_{e} + \sum_{m=1}^{\infty} D_{q_m} \Delta x^{q_m},
\end{align}
where $\{D_q\}$ are the effective coefficients given by
\begin{align}
    D_{q} &\equiv \sum_{n=0}^{\lfloor q/s \rfloor} C_{q-ns,n} \Big( \frac{\Delta t}{\Delta x^s} \Big)^{n}.
\end{align}
To ensure that $D_q$ does not contribute to the convergence rate of $D_q \Delta x^q$, we check whether $D_q$ is constant on all refinement levels. Since $\{C_q\}$ are independent of grid spacing, we only need to check the term $\Delta t / \Delta x^s$. Indeed, on the $l$th refinement level,
\begin{equation}
    \frac{(r_t)^l \Delta t}{[(r_x)^{l} \Delta x]^s} = \bigg[ \frac{r_t}{(r_x)^s} \bigg]^l \frac{\Delta t}{\Delta x^s} = \frac{\Delta t}{\Delta x^s}.
\end{equation}

We find that the expansion series in this case, Equation \ref{eq:model_path_infty}, is essentially the same as the previous one, Equation \ref{eq:model_infty}. The only difference is that while $q_m$ must previously be an integer, it can be a decimal in this case. In addition, the details of $D_q$ are less important as long as they do not affect the convergence rate of $D_q \Delta x^q$. These results can be generalized to refinement in multiple dimensions.

Therefore, the principles and methods presented so far are generally applicable. The general form of the principal equation is as follows.
\begin{equation} \label{eq:model_path}
    \phi = \Tilde{\phi}_{e} + \sum_{m=1}^{k} D_{p_m} \Delta x^{p_m}
\end{equation}
with the unknowns $\Tilde{\phi}_{e}$ and $\{D_p\}$. The unknowns can be obtained by solving the corresponding system of equations
\begin{align} \label{eq:model_path_system}
    \phi_l = \Tilde{\phi}_{e} + \sum_{m=1}^{k} D_{p_m} (r_x^{l-1} \Delta x)^{p_m}
\end{align}
formed by approximate solutions $\{\phi_l\}$ on systematically-refined grids. The knowledge in Section \ref{subsec:est_DE} -- \ref{subsubsec:realize_correct} applies accordingly.


\subsubsection{Refinement with Constant CFL Number} \label{subsubsec:const_cfl}
Here we show the applicability of POEM when the Courant-Friedrichs-
Lewy (CFL) number \citep{CFL1928} is kept constant ($r_t = r_x$) during refinement. We use the advection problem in Section \ref{subsubsec:realize_correct} as a test case, where the CFL number is given by $a \Delta t / \Delta x$. The partial differential equation is discretized with the RK2U2 scheme, and the refinement ratios $r_t = r_x = 0.5$ are used. Starting with Equation \ref{eq:model_RK2U2_infty}, we can show that the RK2U2 scheme is formally second-order-accurate in this case. Therefore, we use the principal equation in Equation \ref{eq:model_path_infty} with $k = 2, p_1 = 2, p_2 = 3$:
\begin{equation} \label{eq:model_path_k2}
    \phi = \Tilde{\phi}_{e} + D_{2}\Delta x^{2} + D_{3}\Delta x^{3}.
\end{equation}
By solving the corresponding system of equations (Equation \ref{eq:model_path_system}), we obtain $\{\Tilde{\phi}_{e}, D_{2}\Delta x^{2}, D_{3}\Delta x^{3}\}$. After that, we calculate the estimated DE, $\Tilde{\varepsilon}$, and the actual DE, $\varepsilon$, using Equations \ref{eq:est_DE} and \ref{eq:DE}, respectively.

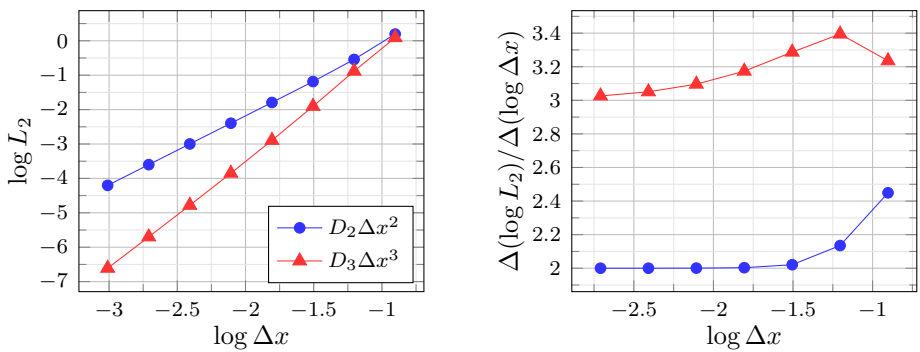
\begin{figure}[!htb]
\centering
\begin{tikzpicture}
\pgfplotstableread[skip first n=2]{figures/const_cfl/cNorm-1D-RK2U2-c-cfl.dat}{\table}
\begin{axis}[
    xtick distance = 0.5,
    ytick distance = 1,
    xlabel = $\log \Delta x$,
    ylabel = $\log L_2$,
    legend pos = south east
]
\addplot[color=blue!80, mark=*, mark size=2] table [x index=0, y index=1] {\table};
\addplot[color=red!80, mark=triangle*, mark size=3] table [x index=0, y index=2] {\table};
\legend{$D_{2} \Delta x^{2}$, $D_{3} \Delta x^{3}$}
\end{axis}
\end{tikzpicture}
\hskip 20pt
\begin{tikzpicture}
\pgfplotstableread{figures/const_cfl/cNorm_slope-1D-RK2U2-c-cfl.dat}{\table}
\begin{axis}[
    xtick distance = 0.5,
    ytick distance = 0.2,
    xlabel = $\log \Delta x$,
    ylabel = $\Delta (\log L_2) / \Delta (\log \Delta x)$
]
\addplot[color=blue!80, mark=*, mark size=2] table [x index=0, y index=1] {\table};
\addplot[color=red!80, mark=triangle*, mark size=3] table [x index=0, y index=2] {\table};
\end{axis}
\end{tikzpicture}
\caption{The $L_2$-norm of effective coefficient terms and their convergence rates obtained by using the correct orders, 2 and 3, in the principal equation when the dimensions $x$ and $t$ are refined along the path of a constant CFL number. Approximate solutions are obtained by solving the one-dimensional advection equation using the RK2U2 scheme.}
\label{fig:cNorm-1D-RK2U2-c-cfl}
\end{figure}

\begin{figure}[!htb]
\centering
\begin{tikzpicture}
\pgfplotstableread[skip first n=2]{figures/const_cfl/discErr-1D-RK2U2-c-cfl.dat}{\table}
\begin{axis}[
    xtick distance = 0.5,
    ytick distance = 1,
    xlabel = $\log \Delta x$,
    ylabel = $\log L$,
    legend pos = south east
]
\addplot[color=blue!80, mark=*, mark size=1] table [x index=0, y index=1] {\table};
\addplot[color=red!80, mark=triangle, mark size=4] table [x index=0, y index=2] {\table};
\addplot[color=teal, mark=square, mark size=4] table [x index=0, y index=2] {\table};
\legend{$L_1$, $L_2$, $L_{\infty}$}
\end{axis}
\end{tikzpicture}
\hskip 20pt
\begin{tikzpicture}
\pgfplotstableread{figures/const_cfl/discErr_slope-1D-RK2U2-c-cfl.dat}{\table}
\begin{axis}[
    minor tick num = 3,
    xtick distance = 0.5,
    ytick distance = 0.08,
    xlabel = $\log \Delta x$,
    ylabel = $\Delta (\log L) / \Delta (\log \Delta x)$
]
\addplot[color=blue!80, mark=*, mark size=1] table [x index=0, y index=1] {\table};
\addplot[color=red!80, mark=triangle, mark size=4] table [x index=0, y index=2] {\table};
\addplot[color=teal, mark=square, mark size=4] table [x index=0, y index=2] {\table};
\end{axis}
\end{tikzpicture}
\caption{Different norms of $\varepsilon$ and their convergence rates in addition to the results in Figure \ref{fig:cNorm-1D-RK2U2-c-cfl}.}
\label{fig:discErr-1D-RK2U2-c-cfl}
\end{figure}
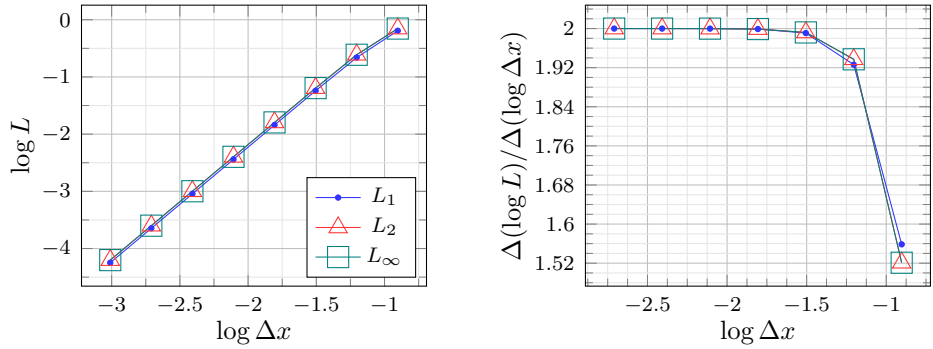

We first examine the results of the coefficient terms. The right figure in Figure \ref{fig:cNorm-1D-RK2U2-c-cfl} suggests that the preset orders 2 and 3 are correct, and the left figure suggests that the approximate solutions lie in the asymptotic range when $\log \Delta x < -2.7$, following the criterion in Equation \ref{eq:beta_tilde}. These implications can be confirmed by the graph of $\varepsilon$ in Figure \ref{fig:discErr-1D-RK2U2-c-cfl}: the convergence rate of $\varepsilon$ deviates from the formal order of 2 by less than 1\% when $\log \Delta x < -1.5$.

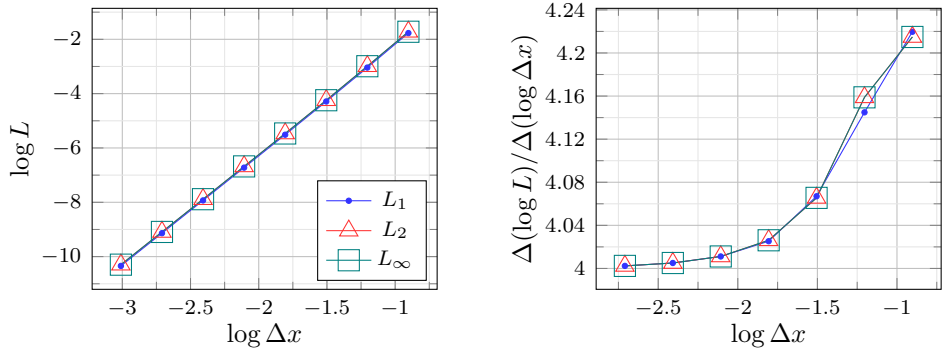
\begin{figure}[!htb]
\centering
\begin{tikzpicture}
\pgfplotstableread[skip first n=2]{figures/const_cfl/es_ex-1D-RK2U2-c-cfl.dat}{\table}
\begin{axis}[
    xtick distance = 0.5,
    ytick distance = 2,
    xlabel = $\log \Delta x$,
    ylabel = $\log L$,
    legend pos = south east
]
\addplot[color=blue!80, mark=*, mark size=1] table [x index=0, y index=1] {\table};
\addplot[color=red!80, mark=triangle, mark size=4] table [x index=0, y index=2] {\table};
\addplot[color=teal, mark=square, mark size=4] table [x index=0, y index=2] {\table};
\legend{$L_1$, $L_2$, $L_{\infty}$}
\end{axis}
\end{tikzpicture}
\hskip 20pt
\begin{tikzpicture}
\pgfplotstableread{figures/const_cfl/es_ex_slope-1D-RK2U2-c-cfl.dat}{\table}
\begin{axis}[
    xtick distance = 0.5,
    ytick distance = 0.04,
    xlabel = $\log \Delta x$,
    ylabel = $\Delta (\log L) / \Delta (\log \Delta x)$
]
\addplot[color=blue!80, mark=*, mark size=1] table [x index=0, y index=1] {\table};
\addplot[color=red!80, mark=triangle, mark size=4] table [x index=0, y index=2] {\table};
\addplot[color=teal, mark=square, mark size=4] table [x index=0, y index=2] {\table};
\end{axis}
\end{tikzpicture}
\caption{Different norms of $\Tilde{\phi}_{e} - \phi_{e}$ and their convergence rates in addition to the results in Figure \ref{fig:cNorm-1D-RK2U2-c-cfl}.}
\label{fig:es_ex-1D-RK2U2-c-cfl}
\end{figure}

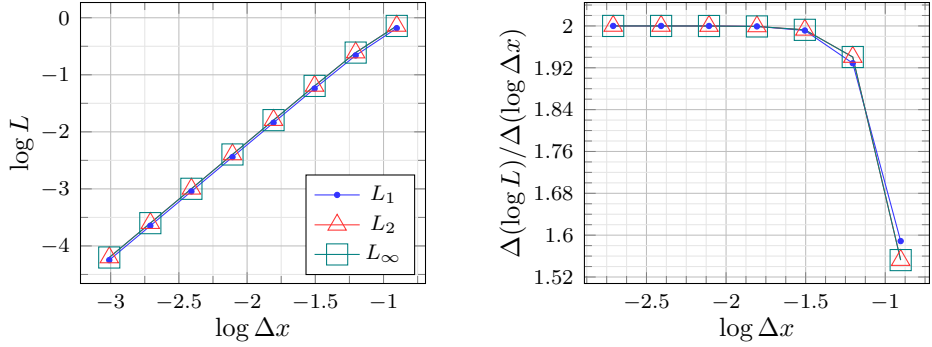
\begin{figure}[!htb]
\centering
\begin{tikzpicture}
\pgfplotstableread[skip first n=2]{figures/const_cfl/esErr-1D-RK2U2-c-cfl.dat}{\table}
\begin{axis}[
    xtick distance = 0.5,
    ytick distance = 1,
    xlabel = $\log \Delta x$,
    ylabel = $\log L$,
    legend pos = south east
]
\addplot[color=blue!80, mark=*, mark size=1] table [x index=0, y index=1] {\table};
\addplot[color=red!80, mark=triangle, mark size=4] table [x index=0, y index=2] {\table};
\addplot[color=teal, mark=square, mark size=4] table [x index=0, y index=2] {\table};
\legend{$L_1$, $L_2$, $L_{\infty}$}
\end{axis}
\end{tikzpicture}
\hskip 20pt
\begin{tikzpicture}
\pgfplotstableread{figures/const_cfl/esErr_slope-1D-RK2U2-c-cfl.dat}{\table}
\begin{axis}[
    minor tick num = 3,
    xtick distance = 0.5,
    ytick distance = 0.08,
    xlabel = $\log \Delta x$,
    ylabel = $\Delta (\log L) / \Delta (\log \Delta x)$
]
\addplot[color=blue!80, mark=*, mark size=1] table [x index=0, y index=1] {\table};
\addplot[color=red!80, mark=triangle, mark size=4] table [x index=0, y index=2] {\table};
\addplot[color=teal, mark=square, mark size=4] table [x index=0, y index=2] {\table};
\end{axis}
\end{tikzpicture}
\caption{Different norms of $\Tilde{\varepsilon}$ and their convergence rates in addition to the results in Figure \ref{fig:cNorm-1D-RK2U2-c-cfl}.}
\label{fig:app_es-1D-RK2U2-c-cfl}
\end{figure}

We also compare the estimated exact solution $\Tilde{\phi}_e$ with the exact solution $\phi_e$ to examine the accuracy of $\Tilde{\phi}_e$. Figure \ref{fig:es_ex-1D-RK2U2-c-cfl} shows that the error norms of $\Tilde{\phi}_e$ converge at a rate of 4, which is expected since the preset orders have proved correct. Furthermore, these error norms are between $10^{-2}$ and $10^{-10}$, which are much lower than those of $\phi_l$ shown in Figure \ref{fig:discErr-1D-RK2U2-c-cfl}. This superior accuracy of $\Tilde{\phi}_e$ suggests that the estimated DE, $\Tilde{\varepsilon}$, is close to the actual DE, $\varepsilon$. Indeed, we find excellent agreement between them in Figure \ref{fig:app_es-1D-RK2U2-c-cfl} and \ref{fig:discErr-1D-RK2U2-c-cfl}.


\FloatBarrier 
\subsubsection{Refinement with Constant Diffusion Number}
Here we show the applicability of POEM when the diffusion number is kept constant ($r_t = r_x^2$) during refinement. The test problem is solving the advection-diffusion equation with a source
\begin{equation}
    \frac{\partial \phi}{\partial t} + a \frac{\partial \phi}{\partial x} = \nu \frac{\partial \phi}{\partial x^2} + 4\pi^2 \nu^2 \cos[2\pi (x - at)]
\end{equation}
over the periodic domain $x \in [0,1]$ with the initial condition $2 + \cos(2\pi x)$. This problem has an analytical solution $\phi(x,t) = 2 + \cos[2\pi (x - at)]$. We are concerned with the solution at $t=2.5$ with $a=0.4$ and $\nu=0.01$. The time derivative and the advection term are discretized using the RK2U2 scheme, whereas the diffusion term is discretized using the fourth-order-centered approximation. When the diffusion number $\nu \Delta t / \Delta x^2$ is kept constant during refinement, such a discretized model is formally second-order-accurate. Therefore, we use the principal equation in Equation \ref{eq:model_path_k2} and the refinement ratios $r_t = r_x^2 = 0.5$.

\begin{figure}[!htb]
\centering
\begin{tikzpicture}
\pgfplotstableread[skip first n=2]{figures/const_dif/cNorm-1D-RK2U2-d-dif.dat}{\table}
\begin{axis}[
    xtick distance = 0.5,
    ytick distance = 1,
    xlabel = $\log \Delta x$,
    ylabel = $\log L_2$,
    legend pos = south east
]
\addplot[color=blue!80, mark=*, mark size=2] table [x index=0, y index=1] {\table};
\addplot[color=red!80, mark=triangle*, mark size=3] table [x index=0, y index=2] {\table};
\legend{$D_{2} \Delta x^{2}$, $D_{3} \Delta x^{3}$}
\end{axis}
\end{tikzpicture}
\hskip 20pt
\begin{tikzpicture}
\pgfplotstableread{figures/const_dif/cNorm_slope-1D-RK2U2-d-dif.dat}{\table}
\begin{axis}[
    xtick distance = 0.5,
    ytick distance = 0.2,
    xlabel = $\log \Delta x$,
    ylabel = $\Delta (\log L_2) / \Delta (\log \Delta x)$
]
\addplot[color=blue!80, mark=*, mark size=2] table [x index=0, y index=1] {\table};
\addplot[color=red!80, mark=triangle*, mark size=3] table [x index=0, y index=2] {\table};
\end{axis}
\end{tikzpicture}
\caption{The $L_2$-norm of effective coefficient terms and their convergence rates obtained by using the correct orders, 2 and 3, in the principal equation when the dimensions $x$ and $t$ are refined along the path of a constant diffusion number. Approximate solutions are obtained by solving the one-dimensional advection-diffusion equation using the RK2U2 scheme for the time derivative and the advection term and the fourth-order-centered approximation for the diffusion term. }
\label{fig:cNorm-1D-RK2U2-d-dif}
\end{figure}
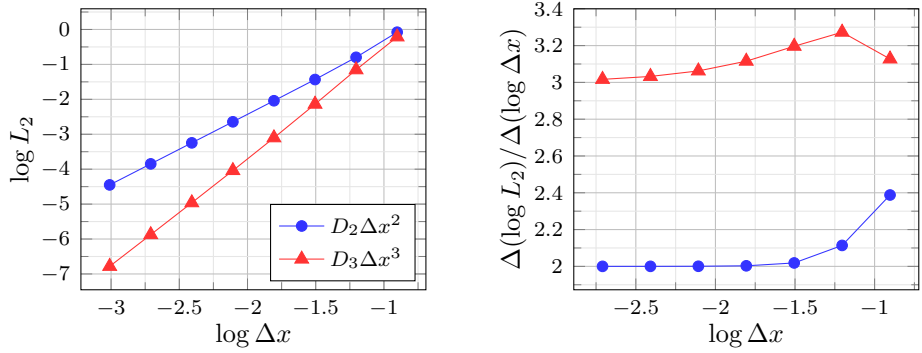

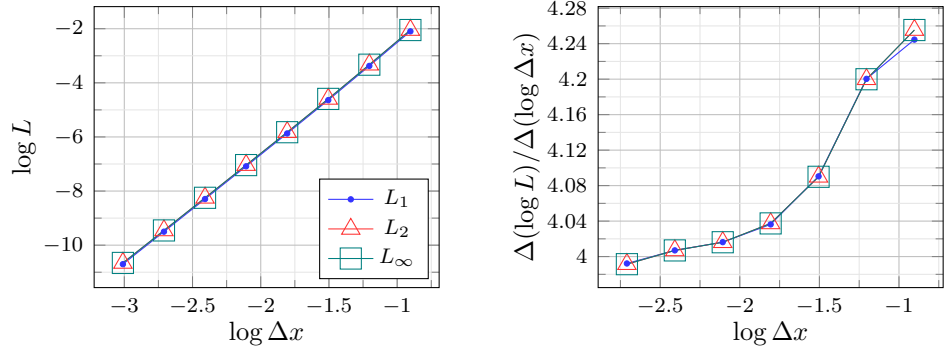
\begin{figure}[!htb]
\centering
\begin{tikzpicture}
\pgfplotstableread[skip first n=2]{figures/const_dif/es_ex-1D-RK2U2-d-dif.dat}{\table}
\begin{axis}[
    xtick distance = 0.5,
    ytick distance = 2,
    xlabel = $\log \Delta x$,
    ylabel = $\log L$,
    legend pos = south east
]
\addplot[color=blue!80, mark=*, mark size=1] table [x index=0, y index=1] {\table};
\addplot[color=red!80, mark=triangle, mark size=4] table [x index=0, y index=2] {\table};
\addplot[color=teal, mark=square, mark size=4] table [x index=0, y index=2] {\table};
\legend{$L_1$, $L_2$, $L_{\infty}$}
\end{axis}
\end{tikzpicture}
\hskip 20pt
\begin{tikzpicture}
\pgfplotstableread{figures/const_dif/es_ex_slope-1D-RK2U2-d-dif.dat}{\table}
\begin{axis}[
    xtick distance = 0.5,
    ytick distance = 0.04,
    xlabel = $\log \Delta x$,
    ylabel = $\Delta (\log L) / \Delta (\log \Delta x)$
]
\addplot[color=blue!80, mark=*, mark size=1] table [x index=0, y index=1] {\table};
\addplot[color=red!80, mark=triangle, mark size=4] table [x index=0, y index=2] {\table};
\addplot[color=teal, mark=square, mark size=4] table [x index=0, y index=2] {\table};
\end{axis}
\end{tikzpicture}
\caption{Different norms of $\Tilde{\phi}_{e} - \phi_{e}$ and their convergence rates in addition to the results in Figure \ref{fig:cNorm-1D-RK2U2-d-dif}.}
\label{fig:es_ex-1D-RK2U2-d-dif}
\end{figure}

The $L_2$-norms of the effective coefficient terms and their convergence rates are shown in Figure \ref{fig:cNorm-1D-RK2U2-d-dif}. With Equations \ref{eq:beta_tilde} and \ref{eq:correctPreset}, we can conclude that the preset orders 2 and 3 are correct and that the asymptotic range begins around $\log \Delta x = -2.7$. The error norms of the estimated exact solution $\Tilde{\phi}_e$ are plotted in Figure \ref{fig:es_ex-1D-RK2U2-d-dif}. We observe the expected convergence rate of 4 and the superior precision of $\Tilde{\phi}_e$.


\FloatBarrier
\subsubsection{Refinement in 2 + 1 Dimensions} \label{subsec:2+1_dim}
In this last example, we show that POEM is applicable when two space dimensions and a time dimension are refined. We demonstrate with the following problem: solving the two-dimensional linear advection equation of $\phi(x,y,t)$,
\begin{equation}
    \frac{\partial \phi}{\partial t} + a_x \frac{\partial \phi}{\partial x} + a_y \frac{\partial \phi}{\partial y} = 0,
\end{equation}
at $t = 2$ over the periodic domain $(x,y) \in [0,1] \cross [0,1]$ with the initial condition $\phi(x,y,0) = 2 + \cos[2\pi (x + y)]$ and advection speeds $a_x = a_y = 0.25$. This equation is discretized using the RK2U2 scheme. We consider the refinement path in which the two CFL numbers $a_x \Delta t / \Delta x$ and $a_y \Delta t / \Delta y$ are constant and use a refinement ratio of 0.5 in all dimensions. Hence, we have $r_t = r_x = r_y = 0.5$. We apply the principal equation in Equation \ref{eq:model_path_k2}. In fact, the principal equation in Equation \ref{eq:model_path} can be recovered with $q_1 = 2$ in this case.

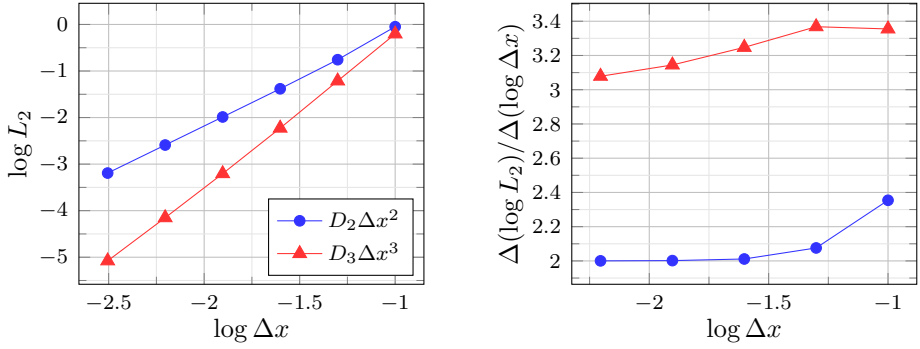
\begin{figure}[!htb]
\centering
\begin{tikzpicture}
\pgfplotstableread[skip first n=2]{figures/extend_2D/cNorm-2D-RK2U2-c.dat}{\table}
\begin{axis}[
    xtick distance = 0.5,
    ytick distance = 1,
    xlabel = $\log \Delta x$,
    ylabel = $\log L_2$,
    legend pos = south east
]
\addplot[color=blue!80, mark=*, mark size=2] table [x index=0, y index=1] {\table};
\addplot[color=red!80, mark=triangle*, mark size=3] table [x index=0, y index=2] {\table};
\legend{$D_{2} \Delta x^{2}$, $D_{3} \Delta x^{3}$}
\end{axis}
\end{tikzpicture}
\hskip 20pt
\begin{tikzpicture}
\pgfplotstableread{figures/extend_2D/cNorm_slope-2D-RK2U2-c.dat}{\table}
\begin{axis}[
    xtick distance = 0.5,
    ytick distance = 0.2,
    xlabel = $\log \Delta x$,
    ylabel = $\Delta (\log L_2) / \Delta (\log \Delta x)$
]
\addplot[color=blue!80, mark=*, mark size=2] table [x index=0, y index=1] {\table};
\addplot[color=red!80, mark=triangle*, mark size=3] table [x index=0, y index=2] {\table};
\end{axis}
\end{tikzpicture}
\caption{The $L_2$-norm of effective coefficient terms and their convergence rates obtained by using the correct orders, 2 and 3, in the principal equation when the dimensions $x$, $y$, and $t$ are refined along the path of constant CFL numbers. The approximate solutions are obtained by solving the two-dimensional advection equation using the RK2U2 scheme.}
\label{fig:cNorm-2D-RK2U2-c}
\end{figure}

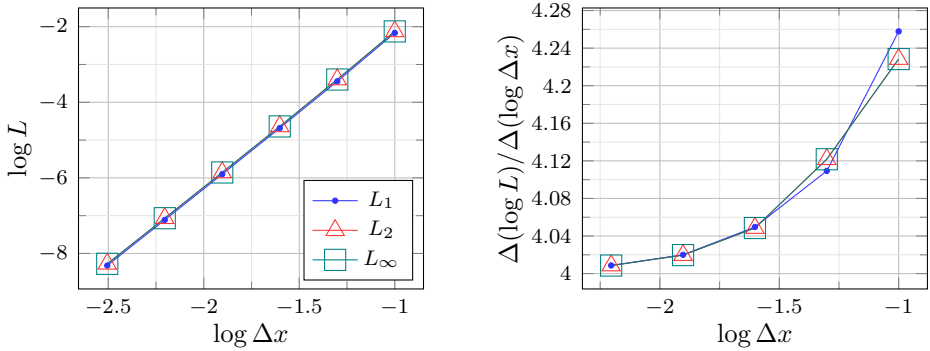
\begin{figure}[!htb]
\centering
\begin{tikzpicture}
\pgfplotstableread[skip first n=2]{figures/extend_2D/es_ex_DE-2D-RK2U2-c.dat}{\table}
\begin{axis}[
    xtick distance = 0.5,
    ytick distance = 2,
    xlabel = $\log \Delta x$,
    ylabel = $\log L$,
    legend pos = south east
]
\addplot[color=blue!80, mark=*, mark size=1] table [x index=0, y index=1] {\table};
\addplot[color=red!80, mark=triangle, mark size=4] table [x index=0, y index=2] {\table};
\addplot[color=teal, mark=square, mark size=4] table [x index=0, y index=2] {\table};
\legend{$L_1$, $L_2$, $L_{\infty}$}
\end{axis}
\end{tikzpicture}
\hskip 20pt
\begin{tikzpicture}
\pgfplotstableread{figures/extend_2D/es_ex_DE_slope-2D-RK2U2-c.dat}{\table}
\begin{axis}[
    xtick distance = 0.5,
    ytick distance = 0.04,
    xlabel = $\log \Delta x$,
    ylabel = $\Delta (\log L) / \Delta (\log \Delta x)$
]
\addplot[color=blue!80, mark=*, mark size=1] table [x index=0, y index=1] {\table};
\addplot[color=red!80, mark=triangle, mark size=4] table [x index=0, y index=2] {\table};
\addplot[color=teal, mark=square, mark size=4] table [x index=0, y index=2] {\table};
\end{axis}
\end{tikzpicture}
\caption{Different norms of $\Tilde{\phi}_{e} - \phi_{e}$ and their convergence rates in addition to the results in Figure \ref{fig:cNorm-2D-RK2U2-c}.}
\label{fig:es_ex-2D-RK2U2-c}
\end{figure}

The effective coefficient terms and the error of the estimated exact solution are plotted in Figure \ref{fig:cNorm-2D-RK2U2-c} and \ref{fig:es_ex-2D-RK2U2-c} respectively. In particular, the estimated exact solution shows a higher order of accuracy and a superior accuracy.


\subsection{Computational Demand}
Like many other grid refinement methods, POEM itself has a relatively low computational cost compared with the simulations calculating the approximate solutions required. However, the requirement of an approximate solution per refinement level may lead to a significant computational cost. Suppose the computational grid is refined along $d$ dimensions with a refinement ratio $r$. Then the number of grid points increases by $1/r^d$ times with $r$. Thus, the commonly used grid doubling ($r = 0.5$) strategy corresponds to an increase rate of $2^d$. The resulting large number of grid points would pose a huge computational burden to the simulations.

A less expensive strategy is to use refinement ratios greater than 0.5, which is compatible with POEM. However, such a strategy involves complications in finding the common locations of the grid points on all refinement levels. This would restrict us to estimating local errors at critical locations of the domain \citep{Roache1994Perspective:Studies,Trivedi2019}. In the next section, we try to resolve those complications so that POEM can be applied to the entire domain.


\FloatBarrier
\section{Method of Interpolating Differences between Approximate Solutions (MIDAS)} \label{sec:MIDAS}
In this section, we present a method called MIDAS to reduce the overall computational cost when POEM is applied. We first study the distribution of shared grid points in the domain when fractional refinement ratios are used. Then, we address its relation to the computational cost and propose how to increase the number of shared grid points. Lastly, we show that MIDAS does not have a significant effect on the estimated DE.


\subsection{Distribution of Shared Grid Points in Fractional Refinement}
When we try to apply POEM to local system quantities over the entire domain, we need to know the distribution of the \textit{shared grid points}, i.e. the common locations of grid points on different refinement levels. In the case of grid doubling ($r=0.5$), a shared grid point is located at each grid point on the coarser grid and alternately resides on the finer grid. However, in the case of refinement ratios greater than 0.5, this distribution is not obvious.

To simplify the problem, we restricted ourselves to the study of systematic grid refinement for Cartesian grids. This means that the grid spacing is uniform on each refinement level and that the same refinement ratio is applied throughout the domain \citep{Roy2010}. In addition, we introduce the concept of an \textit{irreducible unit} which is defined in Definition \ref{def:IU} and illustrated in Figure \ref{fig:setG}.
\begin{definition} \label{def:IU}
    Let $G$ be a set of systematically-refined grids with uniform grid spacing. An irreducible unit of $G$ is the smallest repeating unit in $G$.
\end{definition}

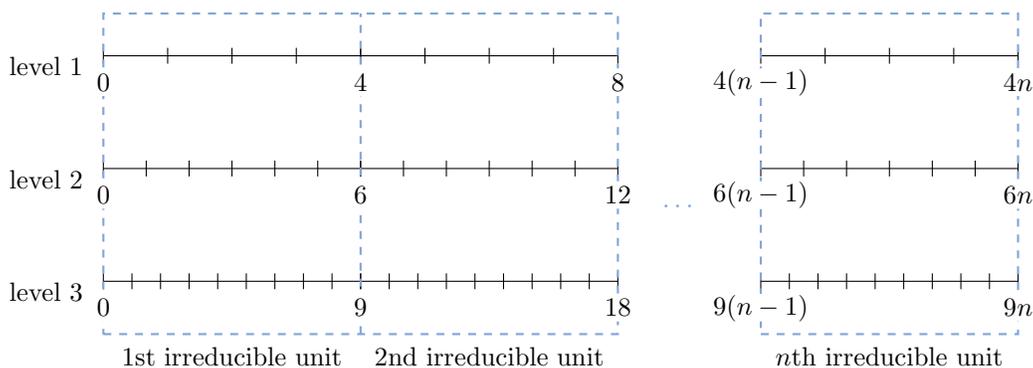
\begin{figure}[!htb]
\centering
\begin{tikzpicture}[
    xscale=0.38,
    font=\footnotesize,
    bx/.style={color=blue!70!green!50, thick, dashed}
]
    \foreach \st in {0,9,23}{
        \draw (\st, 3.0) -- (\st + 9, 3.0);
        \draw (\st, 1.5) -- (\st + 9,1.5);
        \draw (\st, 0.0) -- (\st + 9,0.0);
    
        \foreach \x in {0,...,4}{
            \draw ({\st + \x*9/4}, 3.0cm - 3pt) -- ({\st + \x*9/4}, 3.0cm + 3pt);
        }
        \foreach \x in {0,...,6}{
            \draw ({\st + \x*9/6}, 1.5cm - 3pt) -- ({\st + \x*9/6}, 1.5cm + 3pt);
        }
        \foreach \x in {0,...,9}{
            \draw ({\st + \x}, -3pt) -- ({\st + \x}, 3pt);
        }
    }
    
    \node[bx] at (20cm + 5pt, 1.0) {\ldots};
    
    \draw[bx] (0, 0cm - 20pt) rectangle (18, 3.0cm + 16pt);
    \draw[bx] (9, 0cm - 20pt) -- (9, 3.0cm + 16pt);
    \draw[bx] (23, 0cm - 20pt) rectangle (23 + 9, 3.0cm + 16pt);
    
    \foreach \st in {0,9,18,23,32}{
        \foreach \y in {3.0,1.5,0.0}{
            \draw[color=white, fill=white] (\st*1cm - 5pt, \y*1cm - 10pt - 5pt) rectangle (\st*1cm + 5pt, \y*1cm - 10pt + 5pt);
        }
    }
    
    \foreach \x in {0,4,8}{
        \node at (\x*9/4, 3.0cm - 10pt) {$\x$};
    }
    \foreach \x in {0,6,12}{
        \node at (\x*9/6, 1.5cm - 10
        pt) {$\x$};
    }
    \foreach \x in {0,9,18}{
        \node at (\x, -10pt) {$\x$};
    }
    \node at (23 + 0, {3.0 cm - 10pt}) {$4(n-1)$};
    \node at (23 + 9, {3.0 cm - 10pt}) {$4n$};
    \node at (23 + 0, {1.5 cm - 10pt}) {$6(n-1)$};
    \node at (23 + 9, {1.5 cm - 10pt}) {$6n$};
    \node at (23 + 0, {0.0 cm - 10pt}) {$9(n-1)$};
    \node at (23 + 9, {0.0 cm - 10pt}) {$9n$};
    
    \node at ({0 + 4.5}, -1) {1st irreducible unit};
    \node at ({9 + 4.5}, -1) {2nd irreducible unit};
    \node at ({23 + 4.5}, -1) {$n$th irreducible unit};
    
    \node at (-2, 3.0cm - 4pt) {level 1};
    \node at (-2, 1.5cm - 4pt) {level 2};
    \node at (-2, 0.0cm - 4pt) {level 3};
    
\end{tikzpicture}
\caption{Schematic diagram of $G$ consisting of three grid levels characterized by a refinement ratio of $\frac{2}{3}$. The margin of irreducible units are depicted by dashed lines.}
\label{fig:setG}
\end{figure}

Definition \ref{def:IU} is made in such a way that a shared grid point resides only at the margins of an irreducible unit, thus reducing the complexity of the problem. Nevertheless, an irreducible unit is still hypothetical up to this point. We now propose the following.

\begin{proposition} \label{prop:IU_u&e}
    Let $S_l$ be the number of grid segments on the $l$th level of $G$, and let $s_l$ be the number of grid segments on the $l$th level of the irreducible unit. Also, denote $gcd(\cdot)$ the greatest common divisor of a set of integers. Then an irreducible unit of $G$ exists and is unique. Furthermore, the irreducible unit repeats $gcd(\{S_l\})$ times in $G$ and $s_l = S_l/gcd(\{S_l\})$.
\end{proposition}
\begin{proof}
    From the theory of numbers, the greatest common divisor of two integers exists and is unique \citep{Tattersall2005}. This is also true for multiple integers since $gcd(S_1,S_2,S_3) = gcd(gcd(S_1,S_2),S_3)$. So is $gcd(\{S_l\})$. Therefore, $G$ can be divided into at most $gcd(\{S_l\})$ identical units, which are exactly the irreducible units. Furthermore, every unit contains $S_l/gcd(\{S_l\})$ segments at the $l$th level.
\end{proof}

From Definition \ref{def:IU} and Proposition \ref{prop:IU_u&e}, it follows that if two sets of grids contain the same number of refinement levels and are characterized by the same refinement ratio, then they are identified by the same irreducible unit. In addition, the knowledge that the irreducible unit repeats in a simulation domain is useful for programming: operations over the domain can be implemented by applying the operations on an irreducible unit in each iteration and advancing in a step size of the irreducible unit.

With the concept of irreducible units, we can examine what forms of $r$ maximize the proportion of shared grid points in $G$. Consider a set of two-level irreducible units characterized by various $\{s_1,s_2\}$. Since a shared grid point resides only at the margins, these units must contain the same number of shared grid points. What affects the proportion is the number of interior grid points. Since any additional point in the interior grid will lower the proportion, the lowest proportion is attained when the next refinement level contains one more point in the interior grid than the current level, i.e. $s_2 = s_1 + 1$. Therefore, to maximize the proportion of shared grid points, one should choose refinement ratios of the form $r = s_1 / (s_1 + 1)$, which are referred to as \textit{fractional refinement ratios}. We can further deduce that the grid doubling, where $r=0.5$ and $s_1=1$, corresponds to the maximum proportion.


\subsection{Creating Shared Grid Points by Interpolation} \label{subsec:interp}
If fractional refinement is applied, the targets of POEM, $\Tilde{\phi}_{e}$ and $\{C_{p}h^p\}$, will be sparser in the domain compared to grid doubling. Consequently, the estimated DE will contain a larger statistical error due to the smaller sample size, i.e. $N$ in Equation \ref{eq:L1_norm} -- \ref{eq:L8_norm}. While it seems to be a trade-off between the computational cost and the confidence of the results, MIDAS can be used to compensate the increased uncertainty by adding shared grid points to the domain.

According to Proposition \ref{prop:IU_u&e}, the boundaries of an irreducible unit are the only places where $\phi$ is well defined on all the grid levels and therefore where $\Tilde{\phi}_e$ and $\{C_{p}h^p\}$ can be calculated by POEM. The goal of MIDAS is to obtain $\Tilde{\phi}_e$ and $\{C_{p}h^p\}$ additionally at the locations in the interior where $\phi$ is defined on at least one of the levels. For simplicity we call them the \textit{objective locations}.

The basis of MIDAS, the development of which is inspired by the completed Richardson extrapolation \citep{Roache1993,Richards1997}, is to operate on the differences between approximate solutions
\begin{equation} \label{eq:eij_infty}
    e_{ij} = \phi_i - \phi_j = \sum_{m=1}^{\infty} C_{q_m} h^{q_m} \big[ r^{(i-j)q_m} - 1 \big] r^{(j-1)q_m}
\end{equation}
which eliminates $\Tilde{\phi}_{e}$ in Equation \ref{eq:model_infty}. By these means, a linear operation on $e_{ij}$ applies to all the coefficients $\{C_p\}$ simultaneously. By interpolating $\{e_{ij}\}$ at existing shared grid points to the objective locations, we form a system of equations in terms of the known $\{e_{ij}\}$ and unknown $\{C_p h^p\}$ at the objective locations. By solving this system we obtain $\{C_p h^p\}$, and by subtracting them from $\phi$ we obtain $\Tilde{\phi}_{e}$ (see Equation \ref{eq:model_infty}).

For clarity, we explain this approach with an example. Let us consider the simple but non-trivial case that $G$ comprises three grids characterized by a refinement ratio of $\frac{2}{3}$. By Proposition \ref{prop:IU_u&e}, there are $\{s_l\}=\{4,6,9\}$ grid segments in the irreducible unit, as shown in Figure \ref{fig:grids-1D}.

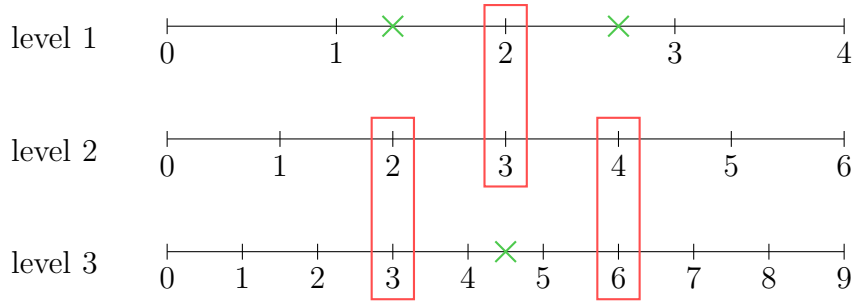
\begin{figure}[!htb]
\centering
\begin{tikzpicture}[
    shared/.style={color=red!70, thick},
    interp/.style={color=green!70!black!70, style=cross out, draw, thick, inner sep=0pt, minimum size=7pt}
]
    \draw (0,3.0) -- (9,3.0);
    \draw (0,1.5) -- (9,1.5);
    \draw (0,0.0) -- (9,0.0);
    
    \foreach \x in {0,...,4}{
        \node at ({\x*9/4}, 3.0cm - 10pt) {$\x$};
        \draw ({\x*9/4}, 3.0cm - 3pt) -- ({\x*9/4}, 3.0cm + 3pt);
    }
    \foreach \x in {0,...,6}{
        \node at ({\x*9/6}, 1.5cm - 10pt) {$\x$};
        \draw ({\x*9/6}, 1.5cm - 3pt) -- ({\x*9/6}, 1.5cm + 3pt);
    }
    \foreach \x in {0,...,9}{
        \node at (\x, -10pt) {$\x$};
        \draw (\x, -3pt) -- (\x, 3pt);
    }
    
    \draw[shared] (4.5cm - 8pt, 1.5cm - 18pt) rectangle (4.5cm + 8pt, 3.0cm + 8pt);
    \draw[shared] (3.0cm - 8pt, 0.0cm - 18pt) rectangle (3.0cm + 8pt, 1.5cm + 8pt);
    \draw[shared] (6.0cm - 8pt, 0.0cm - 18pt) rectangle (6.0cm + 8pt, 1.5cm + 8pt);
    
    \node[interp] at (3.0,3) {};
    \node[interp] at (6.0,3) {};
    \node[interp] at (4.5,0) {};
    
    \node at (-1.5, 3.0cm - 4pt) {level 1};
    \node at (-1.5, 1.5cm - 4pt) {level 2};
    \node at (-1.5, 0.0cm - 4pt) {level 3};
    
\end{tikzpicture}
\caption{The irreducible unit of $G$ in Figure \ref{fig:setG} (dashed box). The numbers refer to the grid point indices of the corresponding refinement level. In addition to the shared grid points at the boundaries that are shared by all the grids, there are grid points in the interior that are shared by two of the grids (red boxes). The latter are the objective locations (green crosses) in this example.}
\label{fig:grids-1D}
\end{figure}

In this example, we choose the objective locations to be where $\phi$ is well defined on two grid levels. They are also where one of $e_{12}$ and $e_{23}$ is well defined. These locations can be found by advancing grid points in a step size of $s_1/gcd(s_1,s_2) = s_2/gcd(s_2,s_3) = 2$ on the coarser grid or $s_2/gcd(s_1,s_2) = s_3/gcd(s_2,s_3) = 3$ on the finer grid in light of Proposition \ref{prop:IU_u&e}.

Let $x \in [0,1]$ be the domain of the irreducible unit. We interpolate $e_{ij}$ at the objective locations $x_o \in \big\{ \frac{1}{3}, \frac{1}{2}, \frac{2}{3} \big\}$ using $e_{ij}$ at the nearest two points $x_a,x_b \in \big\{ 0, \frac{1}{3}, \frac{1}{2}, \frac{2}{3}, 1 \big\}$ with suitable weights $\Gamma_a,\Gamma_b$. This can be formulated as follows.
\begin{align}
    e_{ij}(x_o) &\equiv \Gamma_a e_{ij}(x_a) + \Gamma_b e_{ij}(x_b) \label{eq:eij_extend} \\
    &= \sum_{m=1}^{\infty} \big[ \Gamma_a C_{q_m}(x_a) + \Gamma_b C_{q_m}(x_b) \big] h^{q_m} \big[ r^{(i-j)q_m} - 1 \big] r^{(j-1)q_m} \\
    &= \sum_{m=1}^{\infty} \big[ C_{q_m}(x_a) + \mathcal{O}(h^2) \big] h^{q_m} \big[ r^{(i-j)q_m} - 1 \big] r^{(j-1)q_m},
\end{align}
where 
\begin{equation} \label{eq:C_xo}
    C_{q_m}(x_o) = \Gamma_a C_{q_m}(x_a) + \Gamma_b C_{q_m}(x_b) + \mathcal{O}(h^{2})
\end{equation}
has been used. Using $x_o = \frac{1}{3}$ as an example, we use $x_a = 0$, $x_b = \frac{1}{2}$, $\Gamma_a = \frac{1}{3}$, $\Gamma_b = \frac{2}{3}$. We note that with a linear interpolation the interpolation error is of the order $h^{q_1+2}$; therefore, the coefficient terms of orders $q_1$ and $q_1 + 1$ are not affected. If higher-order coefficient terms are considered, a higher-order interpolation can be used to increase the order of the interpolation error.

Subsequently, we can form a system of equations in terms of $e_{ij}(x_o)$:
\begin{align} \label{eq:sys2eq_x}
    \begin{bmatrix}
        (r^{p_1} - 1) & (r^{p_2} - 1) \\
        (r^{p_1} - 1) r^{p_1} & (r^{p_2} - 1) r^{p_2}
    \end{bmatrix}
    \begin{bmatrix}
        C_{p_1} h^{p_1} \\
        C_{p_2} h^{p_2}
    \end{bmatrix}
    &=
    \begin{bmatrix}
        e_{21} \\
        e_{32}
    \end{bmatrix},
\end{align}
where $\{q_m\}$ are replaced by $\{p_m\}$ following the procedures of POEM.\footnote{In fact, this system of equations can be derived from Equation \ref{eq:sys3eq_t}.} Similarly, $C_{p_1} h^{p_1}$ and $C_{p_2} h^{p_2}$ can be obtained by solving this system. However, $\Tilde{\phi}_{e}$ is obtained by subtracting $C_{p_1} h^{p_1}$ and $C_{p_2} h^{p_2}$ from a well-defined $\phi$, according to Equation \ref{eq:model_orig}.

Now we evaluate the increased proportion of shared grid points due to the use of MIDAS. In the irreducible unit, there are four shared grid points: three created by interpolation and one on the boundary. The other boundary point is neglected because it overlaps with the adjacent irreducible unit. As a result, the proportion of shared grid points has increased from 11\% to 44\% on the finest grid which consumes the most computational resources. In comparison, the proportion is 25\% when grid doubling is applied only. In principle, the proportion can be further increased by interpolating to locations where $\phi$ is well defined at one single grid level.


\subsection{Demonstration of Consistencies} \label{subsec:consistencies}
Here we demonstrate that the results obtained by applications of MIDAS using various refinement ratios are consistent with those obtained by grid doubling.

Since the number of grid segments must be an integer, a grid cannot be refined with a fractional refinement ratio for an arbitrary number of times. For instance, a grid which contains four grid segments can be refined only twice with a ratio of $\frac{2}{3}$ (see Figure \ref{fig:grids-1D}); a further refinement would result in a grid with 13.5 grid segments, which is impossible.

In order to make comparisons in a range of grid resolutions, we employ two refinement schemes: \textit{global} refinement and \textit{local} refinement. While the global refinement ratio is always 0.5, the local refinement ratio is one of $\{\frac{2}{3}, \frac{3}{4}, \frac{4}{5}, \frac{9}{10}\}$, denoted by $r_x$. For each $r_x$, we refine the coarsest grid of the respective irreducible unit twice with the local refinement ratio to form a set of three grids, in which POEM is applied. Then we refine each of them with the global refinement ratio. As a result, the grids generated in terms of the number of grid segments are $\{\{4,6,9\}, \{\{8,12,18\}, \{16,24,36\}, \dots \}$ for the case of $r_x = \frac{2}{3}$. We note that this procedure generates more than sufficient grid points, posing unnecessary computational burdens; a more practical approach is presented in Section \ref{sec:POEM-MIDAS}.

For the test case, we use the advection problem in Section \ref{subsubsec:realize_correct} and the refinement path of constant CFL number. Based on Equations \ref{eq:model_path_k2} and \ref{eq:eij_infty}, we describe the differences between approximate solutions as follows.
\begin{equation}
    e_{ij} = D_{2} \Delta x^{2} \big[ r_x^{2(i-j)} - 1 \big] r_x^{2(j-1)} + D_{3} \Delta x^{3} \big[ r_x^{3(i-j)} - 1 \big] r_x^{3(j-1)}.
\end{equation}

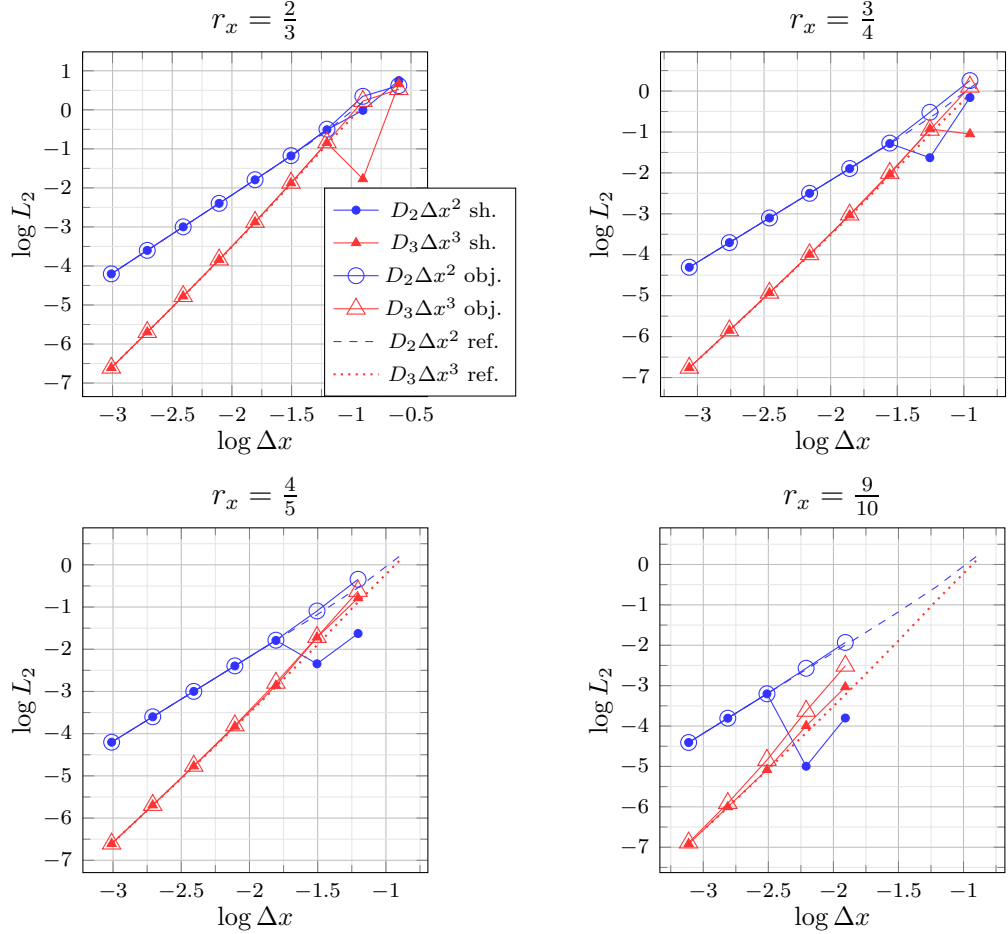
\begin{figure}[!htb]
\centering
\pgfplotstableread[skip first n=2]{figures/const_cfl/cNorm-1D-RK2U2-c-cfl.dat}{\table}
\begin{tikzpicture}
\pgfplotstableread[skip first n=2]{figures/consistencies/cons-1D-RK2U2-c0.67.dat}{\loadedtable}
\pgfplotsset{
    width=0.45\textwidth,
    height=0.45\textwidth
}
\begin{axis}[
    xtick distance = 0.5,
    ytick distance = 1,
    xlabel = $\log \Delta x$,
    ylabel = $\log L_2$,
    legend style={at={(0.7,0)},anchor=south west},
    title = {$r_x=\frac{2}{3}$}
]
\addplot[color=blue!80, mark=*, mark size=1.5] table [x index=0, y index=1] {\loadedtable};
\addplot[color=red!80, mark=triangle*, mark size=2] table [x index=0, y index=2] {\loadedtable};
\addplot[color=blue!80, mark=o, mark size=3] table [x index=0, y index=3] {\loadedtable};
\addplot[color=red!80, mark=triangle, mark size=4] table [x index=0, y index=4] {\loadedtable};
\addplot[color=blue!80, style=dashed] table [x index=0, y index=1] {\table};
\addplot[color=red!80, style=dotted, thick] table [x index=0, y index=2] {\table};
\legend{$D_{2} \Delta x^{2}$ sh., $D_{3} \Delta x^{3}$ sh., $D_{2} \Delta x^{2}$ obj., $D_{3} \Delta x^{3}$ obj., $D_{2} \Delta x^{2}$ ref., $D_{3} \Delta x^{3}$ ref.}
\end{axis}
\end{tikzpicture}
\hskip 20pt
\begin{tikzpicture}
\pgfplotstableread[skip first n=2]{figures/consistencies/cons-1D-RK2U2-c0.75.dat}{\loadedtable}
\pgfplotsset{
    width=0.45\textwidth,
    height=0.45\textwidth
}
\begin{axis}[
    xtick distance = 0.5,
    ytick distance = 1,
    xlabel = $\log \Delta x$,
    ylabel = $\log L_2$,
    legend pos = south east,
    title = {$r_x=\frac{3}{4}$}
]
\addplot[color=blue!80, mark=*, mark size=1.5] table [x index=0, y index=1] {\loadedtable};
\addplot[color=red!80, mark=triangle*, mark size=2] table [x index=0, y index=2] {\loadedtable};
\addplot[color=blue!80, mark=o, mark size=3] table [x index=0, y index=3] {\loadedtable};
\addplot[color=red!80, mark=triangle, mark size=4] table [x index=0, y index=4] {\loadedtable};
\addplot[color=blue!80, style=dashed] table [x index=0, y index=1] {\table};
\addplot[color=red!80, style=dotted, thick] table [x index=0, y index=2] {\table};
\end{axis}
\end{tikzpicture}
\hfill
\begin{tikzpicture}
\pgfplotstableread[skip first n=2]{figures/consistencies/cons-1D-RK2U2-c0.8.dat}{\loadedtable}
\pgfplotsset{
    width=0.45\textwidth,
    height=0.45\textwidth
}
\begin{axis}[
    xtick distance = 0.5,
    ytick distance = 1,
    xlabel = $\log \Delta x$,
    ylabel = $\log L_2$,
    legend pos = south east,
    title = {$r_x=\frac{4}{5}$}
]
\addplot[color=blue!80, mark=*, mark size=1.5] table [x index=0, y index=1] {\loadedtable};
\addplot[color=red!80, mark=triangle*, mark size=2] table [x index=0, y index=2] {\loadedtable};
\addplot[color=blue!80, mark=o, mark size=3] table [x index=0, y index=3] {\loadedtable};
\addplot[color=red!80, mark=triangle, mark size=4] table [x index=0, y index=4] {\loadedtable};
\addplot[color=blue!80, style=dashed] table [x index=0, y index=1] {\table};
\addplot[color=red!80, style=dotted, thick] table [x index=0, y index=2] {\table};
\end{axis}
\end{tikzpicture}
\hskip 53pt 
\begin{tikzpicture}
\pgfplotstableread[skip first n=2]{figures/consistencies/cons-1D-RK2U2-c0.9.dat}{\loadedtable}
\pgfplotsset{
    width=0.45\textwidth,
    height=0.45\textwidth
}
\begin{axis}[
    xtick distance = 0.5,
    ytick distance = 1,
    xlabel = $\log \Delta x$,
    ylabel = $\log L_2$,
    legend pos = south east,
    title = {$r_x=\frac{9}{10}$}
]
\addplot[color=blue!80, mark=*, mark size=1.5] table [x index=0, y index=1] {\loadedtable};
\addplot[color=red!80, mark=triangle*, mark size=2] table [x index=0, y index=2] {\loadedtable};
\addplot[color=blue!80, mark=o, mark size=3] table [x index=0, y index=3] {\loadedtable};
\addplot[color=red!80, mark=triangle, mark size=4] table [x index=0, y index=4] {\loadedtable};
\addplot[color=blue!80, style=dashed] table [x index=0, y index=1] {\table};
\addplot[color=red!80, style=dotted, thick] table [x index=0, y index=2] {\table};
\end{axis}
\end{tikzpicture}
\caption{Comparison between the $L_2$-norm of effective coefficient terms at the originally-shared grid points (sh.) and those at the objective locations (obj.) using different fractional refinement ratios. The test problem is same as that used in Figure \ref{fig:cNorm-1D-RK2U2-c-cfl}. That figure is also plotted here for comparisons (ref.).}
\label{fig:eval-1D-RK2U2}
\end{figure}

The $L_2$-norms of $D_2 \Delta x^2$ and $D_3 \Delta x^3$ for various $r_x$ are shown in Figure \ref{fig:eval-1D-RK2U2}, where each point in the plots is obtained from a set of locally refined grids. The results from grid doubling, Figure \ref{fig:cNorm-1D-RK2U2-c-cfl} (left), is also plotted for comparisons. When comparing the results from the same refinement ratio, we find that the coefficient terms at the originally shared grid points generally agree with those at the objective locations. This confirms that the interpolation errors are insignificant. When comparing the results across refinement ratios, we find that the coefficient terms agree well with each other. This suggests that choosing a different refinement ratio does not cause a significant difference in the results. The larger differences found in the coarser grids can be attributed to the larger statistical error of the norms due to the reduced sample size, that is, $N$ in Equation \ref{eq:L2_norm}.


\FloatBarrier
\section{POEM in Combination with MIDAS} \label{sec:POEM-MIDAS}
We are prepared to incorporate MIDAS into POEM. First, we outline a general procedure for implementing POEM together with MIDAS. Then, we revisit the test case in Section \ref{subsec:2+1_dim} with the additional use of MIDAS.

\subsection{General Procedure} \label{subsec:general_proc}
Suppose we have obtained approximate solutions on a set of systematically-refined grids of uniform grid spacing. Then, we can implement POEM in combination with MIDAS by following this general procedure:
\begin{enumerate}
    \item Identify the irreducible unit based on Proposition \ref{prop:IU_u&e}.
    \item Iterate over the irreducible units in the simulation domain. Perform the following steps in each iteration.
    \item Calculate $e_{ij} = \phi_{i} - \phi_{j}$ at the grids points shared by at least two refinement levels. (See Figure \ref{fig:grids-1D}.)
    \item Apply interpolations on $e_{ij}$ at the objective locations. (See Equation \ref{eq:eij_extend}.)
    \item Construct a model for $e_{ij}$ of the form
    \begin{equation}  \label{eq:eij_path}
        e_{ij} = \sum_{m=1}^{k} D_{p_m} h^{p_m} \big[ r^{(i-j)p_m} - 1 \big] r^{(j-1)p_m}
    \end{equation}
    with preset orders $\{p_m\}$. (See Equation \ref{eq:model_path_system} and \ref{eq:eij_infty}.)
    \item Form a system of equations of $\{e_{ij}\}$ at the objective locations and solve for $\{D_{p_m} h^{p_m}\}$. (See Equation \ref{eq:sys2eq_x}.)
    \item Check whether each $D_{p_m} h^{p_m}$ converges at the rate of $p_m$. If this is true, proceed to the next step; otherwise, go back to the previous step with the wrong orders replaced by $\mu$. (See Section \ref{subsubsec:guarantee}.)
    \item Obtain $\Tilde{\phi}_e$ by subtracting the sum of $\{D_{p_m} h^{p_m}\}$ from $\phi$. (See Equation \ref{eq:model_path}.)
    \item Estimate the DE using Equation \ref{eq:est_DE}. Assess the reliability of the estimate using Equation \ref{eq:beta_tilde}.
\end{enumerate}

\subsection{Revisiting the Refinement in 2+1 Dimensions} \label{subsec:revisit}
Here we revisit the test case in Section \ref{subsec:2+1_dim} with the additional use of MIDAS. We examine if the two results are consistent and independent of refinement ratios. We also evaluate the reduced computational cost compared with the grid doubling approach.

\begin{figure}[!htb]
\centering
\begin{tikzpicture}[
    scale=0.7, font=\small,
    nd/.style={circle, inner sep=0pt, minimum size=5pt},
    color12/.style={fill=orange, color=orange},
    color23/.style={fill=blue!50, color=blue!50},
    colorInt/.style={fill=green!80!black, color=green!80!black}
]

    \begin{scope}
        \myGlobalTransformation{0}{9};
        \fill[black!10, opacity=0.3] (0,0) rectangle (9,9);
        \draw [step=2.25cm, black!50] grid (9,9);
    \end{scope}
    \begin{scope}
        \myGlobalTransformation{0}{4.5};
        \fill[black!10, opacity=0.3] (0,0) rectangle (9,9);
        \draw [step=1.5cm, black!50] grid (9,9);
    \end{scope}
    \begin{scope}
        \myGlobalTransformation{0}{0};
        \fill[black!10, opacity=0.3] (0,0) rectangle (9,9);
        \draw [step=1cm, black!50] grid (9,9);
    \end{scope}
    
    \begin{scope}
        \myGlobalTransformation{0}{4.5};
        \foreach \x in {0,4.5,9} {
            \foreach \y in {0,4.5,9} {
                \node[below] (thisNode) at (\x,\y) {};
                \pgftransformreset
                \draw[color12, thick] (thisNode) -- ++(0,3.3);
            }
        }
    \end{scope}
    \begin{scope}
        \myGlobalTransformation{0}{0};
        \foreach \x in {0,3,6,9} {
            \foreach \y in {0,3,6,9} {
                \node[below] (thisNode) at (\x,\y) {};
                \pgftransformreset
                \draw[color23, thick] (thisNode) -- ++(0,3.3);
            }
        }
    \end{scope}
    
    \begin{scope}
        \myGlobalTransformation{0}{4.5};
        \foreach \x in {0,3,6,9} {
            \foreach \y in {0,3,6,9} {
                \ifthenelse {\(0=\x \AND 0=\y\) \OR 
                            \(9=\x \AND 0=\y\) \OR 
                            \(0=\x \AND 9=\y\) \OR 
                            \(9=\x \AND 9=\y\)}{}{
                    \node[below] (thisNode) at (\x,\y) {};
                    \pgftransformreset
                    \draw[colorInt, very thick, dotted] (thisNode) -- ++(0,3.3);
                }
            }
        }
    \end{scope}
    
    \begin{scope}
        \myGlobalTransformation{0}{9};
        \foreach \x in {0,4.5,9} {
            \foreach \y in {0,4.5,9} {
                \node[nd, color12] at (\x,\y) {};
            }
        }
    \end{scope}
    \begin{scope}
        \myGlobalTransformation{0}{4.5};
        \foreach \x in {0,4.5,9} {
            \foreach \y in {0,4.5,9} {
                \ifthenelse {\(0=\x \AND 0=\y\) \OR 
                            \(9=\x \AND 0=\y\) \OR 
                            \(0=\x \AND 9=\y\) \OR 
                            \(9=\x \AND 9=\y\)}{
                    \node[nd, circle split, circle split part fill={color12,color23}, draw=black!0] at (\x,\y) {};
                }{
                    \node[nd, color12] at (\x,\y) {};
                }
            }
        }
    \end{scope}
    \begin{scope}
        \myGlobalTransformation{0}{4.5};
        \foreach \x in {0,3,6,9} {
            \foreach \y in {0,3,6,9} {
                \ifthenelse {\(0=\x \AND 0=\y\) \OR 
                            \(9=\x \AND 0=\y\) \OR 
                            \(0=\x \AND 9=\y\) \OR 
                            \(9=\x \AND 9=\y\)}{}{
                    \node[nd, color23] at (\x,\y) {};
                }
            }
        }
    \end{scope}
    \begin{scope}
        \myGlobalTransformation{0}{0};
        \foreach \x in {0,3,6,9} {
            \foreach \y in {0,3,6,9} {
                \node[nd, color23] at (\x,\y) {};
            }
        }
    \end{scope}
    
    \begin{scope}
        \myGlobalTransformation{0}{9};
        \foreach \x in {0,3,6,9} {
            \foreach \y in {0,3,6,9} {
                \ifthenelse {\(0=\x \AND 0=\y\) \OR 
                            \(9=\x \AND 0=\y\) \OR 
                            \(0=\x \AND 9=\y\) \OR 
                            \(9=\x \AND 9=\y\)}{}{
                    \node[nd, colorInt, style=cross out, draw, thick] at (\x,\y) {};
                }
            }
        }
    \end{scope}
    
    \begin{scope}
        \myGlobalTransformation{13}{0.5};
        \draw[->, very thick] (0,0) -- (1.5,0) node[anchor=west] {$x$};
        \draw[->, very thick] (0,0) -- (0,1.5) node[anchor=south west] {$y$};
    \end{scope}
    
    \begin{scope}
        \myGlobalTransformation{0}{9};
        \node[left] at (0,0) {(0,0)};
        \node[above left] at (0,9) {(0,4)};
        \node[right] at (9,9) {(4,4)};
    \end{scope}
    \begin{scope}
        \myGlobalTransformation{0}{4.5};
        \node[left] at (0,0) {(0,0)};
        \node[right] at (9,9) {(6,6)};
    \end{scope}
    \begin{scope}
        \myGlobalTransformation{0}{0};
        \node[left] at (0,0) {(0,0)};
        \node[below right] at (9,0) {(9,0)};
        \node[right] at (9,9) {(9,9)};
    \end{scope}
    
    \node at (-1, 9.0cm + 20pt) {level 1};
    \node at (-1, 4.5cm + 20pt) {level 2};
    \node at (-1, 0.0cm + 20pt) {level 3};
    
    \matrix[draw, semithick, column sep=10pt, below] at (current bounding box.south) {
        \node[color12, label=right:shared by level 1 and 2] {}; &
        \node[color23, label=right:shared by level 2 and 3] {}; &
        \node[colorInt, label=right:interpolation] {}; \\
    };
    
\end{tikzpicture}
\caption{The irreducible unit of $G$ that consists of three grids with $r_x=r_y=\frac{2}{3}$. Coordinates $(i,j)$ refer to the indices of the grid points on the respective level. While the orange and blue circles connected with solid lines represent existing shared grid points, the green crosses attached to a dotted line represent the objective locations of MIDAS being implemented here.}
\label{fig:grids-2D}
\end{figure}
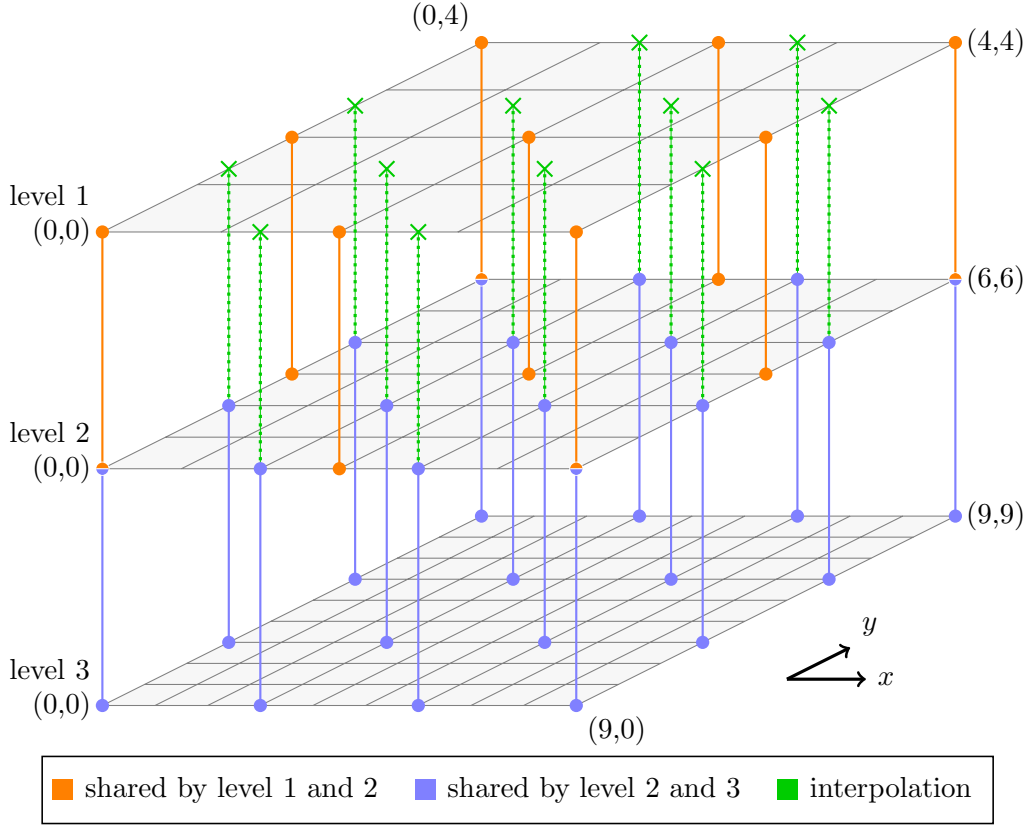

We solve the test problem using the refinement ratios $\{\frac{2}{3}, \frac{3}{4}, \frac{4}{5}\}$. As an illustration of our applications, we consider a set of three systematically refined grids characterized by $r = r_x = r_y = \frac{2}{3}$. The irreducible unit contains grid segments $4 \cross 4$, $6 \cross 6$, and $9 \cross 9$ in the coarse, medium, and fine grid, respectively, as shown in Figure \ref{fig:grids-2D}.

\begin{figure}[!htb]
\centering
\begin{tikzpicture}
\pgfplotsset{
    width=0.6\textwidth,
}
\begin{axis}[
    xtick distance = 0.1,
    ytick distance = 0.5,
    xlabel = $\log \Delta x$,
    ylabel = $\log ||D_{p} \Delta x^{p}||_2$,
    legend pos = outer north east,
    legend entries = {
        $p=2, r=\frac{1}{2}$\\$p=2, r=\frac{2}{3}$\\$p=2, r=\frac{3}{4}$\\$p=2, r=\frac{4}{5}$\\
        $p=3, r=\frac{1}{2}$\\$p=3, r=\frac{2}{3}$\\$p=3, r=\frac{3}{4}$\\$p=3, r=\frac{4}{5}$\\
    }
]
\addplot[color=blue!80, mark=*, mark size=1.5, line width=0.4pt] table [skip first n=2, x index=0, y index=1] {figures/revisit/cNorm-2D-RK2U2-c0.5.dat};
\addplot[color=blue!80, mark=+, mark size=4, line width=0.4pt] table [skip first n=2, x index=0, y index=1] {figures/revisit/cNorm-2D-RK2U2-c0.67.dat};
\addplot[color=blue!80, mark=o, mark size=4, line width=0.4pt] table [skip first n=2, x index=0, y index=1] {figures/revisit/cNorm-2D-RK2U2-c0.75.dat};
\addplot[color=blue!80, mark=diamond, mark size=4, line width=0.4pt] table [skip first n=2, x index=0, y index=1] {figures/revisit/cNorm-2D-RK2U2-c0.8.dat};

\addplot[color=red!80, mark=triangle*, mark size=2, line width=0.4pt] table [skip first n=2, x index=0, y index=2] {figures/revisit/cNorm-2D-RK2U2-c0.5.dat};
\addplot[color=red!80, mark=x, mark size=4, line width=0.4pt] table [skip first n=2, x index=0, y index=2] {figures/revisit/cNorm-2D-RK2U2-c0.67.dat};
\addplot[color=red!80, mark=triangle, mark size=4, line width=0.4pt] table [skip first n=2, x index=0, y index=2] {figures/revisit/cNorm-2D-RK2U2-c0.75.dat};
\addplot[color=red!80, mark=square, mark size=4, line width=0.4pt] table [skip first n=2, x index=0, y index=2] {figures/revisit/cNorm-2D-RK2U2-c0.8.dat};
\end{axis}
\end{tikzpicture}
\caption{The results of a real application of POEM and MIDAS in solution verification using different refinement ratios. The problem setting is that described in Figure \ref{fig:cNorm-2D-RK2U2-c}; however, here the refinements start with a much finer grid.}
\label{fig:compare_ratios}
\end{figure}
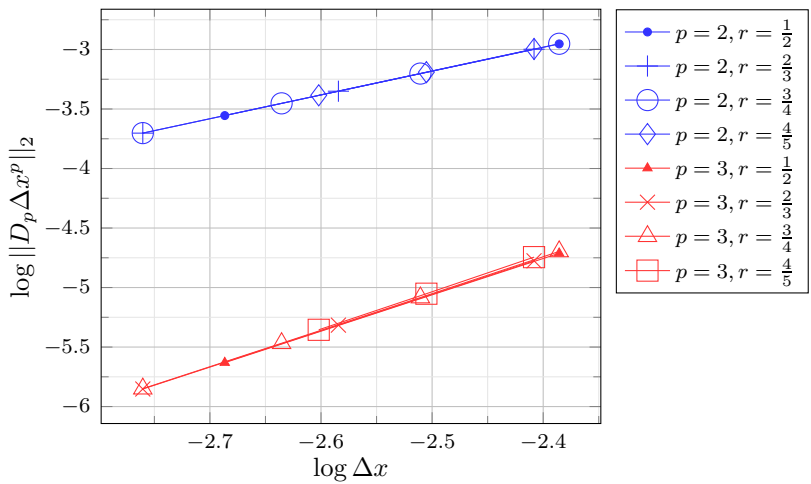

We choose the objective locations of MIDAS to be where the medium and fine grids have shared grid points. To this end, we need to extend the definition of $e_{ij}$ to the places marked by a green cross in Figure \ref{fig:grids-2D}. For the objective locations on an edge, we apply the two-point interpolation given in Equation \ref{eq:C_xo}. For those in the interior, we apply a four-point linear interpolation with the four nearest shared grid points and suitable weights $\{\Gamma_i\}$:
\begin{align}
    e_{ij}(x_o,y_o) &\equiv \Gamma_{a}e_{ij}(x_a,y_a) + \Gamma_{b}e_{ij}(x_b,y_b) + \Gamma_{c}e_{ij}(x_c,y_c) + \Gamma_{d}e_{ij}(x_d,y_d).
\end{align}
To describe $\{e_{ij}\}$ in terms of coefficient terms, we use Equation \ref{eq:eij_path} with $k=2$.

We use a single refinement ratio, in contrast to a global and a local refinement ratio in Section \ref{subsec:consistencies}, to mimic real applications of fractional refinement. The number of grid segments in the coarsest grid is 243 for $r \in \{\frac{1}{2}, \frac{3}{4}\}$ and 256 for $r \in \{\frac{2}{3}, \frac{4}{5}\}$ in both $x$ and $y$ dimensions.

The resulting $L_2$-norms of the effective coefficient terms are plotted in Figure \ref{fig:compare_ratios}. The results from grid doubling are also plotted in this figure for comparisons. We find excellent agreement among these results. This implies that the additional use of MIDAS with different refinement ratios does not affect the results of POEM.

Last but not least, we compare the computational cost of applying fractional refinement and MIDAS with that of applying solely grid doubling. We conducted the experiment serially on an Intel Core i5-7200U (2.5GHz, 3MB cache) with 4GB of memory using double precision. We repeat the simulation of $r=0.5$ and $r=0.75$ 20 times each to collect run-time statistics. The average run-time is $(1474 \pm 9)$ s for $r=0.5$ and $(324 \pm 1)$ s for $r=0.75$ in the form of a 95\% confidence interval. Hence, we have achieved a speed-up of 4.55 times in this problem.


\FloatBarrier
\section{Conclusion}
POEM exploits the fact that the orders of coefficient terms depend solely on the numerical scheme being used, whereas the coefficients of these terms depend additionally on the numerical problem being solved. Therefore, this method calculates the coefficients instead of the orders of the coefficient terms. With the procedure described in this paper, the user is guaranteed to obtain the correct orders by repeated use of this method. Consequently, the estimated exact solution achieves a high order of accuracy and produces an accurate estimate of the discretization error. In addition, this method allows for a direct comparison of the coefficient terms to assess the asymptotic convergence of the approximate solutions, which is fundamental to the reliability of the estimated discretization error.

POEM requires a lower computational cost when the refinement ratio is higher. However, the estimated error suffers from higher uncertainty due to the reduced number of shared grid points. We show that the proportion of shared grid points attains maximum if the refinement ratios are in a specific form of fractions. Furthermore, we introduce additional shared grid points using MIDAS, which exploits the linearity of coefficient terms to interpolate them at once. As a result, POEM becomes applicable to interior grid points not shared with all refinement levels. Therefore, we can obtain a global estimate of the discretization error of lower uncertainty at a reduced computational cost. Although we focus on the implementations of POEM and MIDAS for Cartesian grids, these methods are directly applicable to other grids which can be transformed into a Cartesian grid.


\section*{Acknowledgement}
This project would not have been possible without the guidance of Dr. Andreas Haselbacher. I also would like to recognize the invaluable assistance of Vera H. S. Wu, M.Sc., for reviewing the manuscript.



\bibliographystyle{elsarticle-num-names} 
\bibliography{ref-short,references}


\end{document}